\documentclass[11pt,reqno]{amsart}
\usepackage[margin= 2.0cm]{geometry}
\usepackage{algorithm}
\usepackage{algpseudocode}
\usepackage{setspace}
\usepackage{amsmath}
\usepackage{amsfonts}
\usepackage{amssymb}
\usepackage{amsthm}
\usepackage{mathtools} % An improvement of amsmath
\usepackage{latexsym}
\usepackage{enumerate}
\usepackage{cancel}
\usepackage{ragged2e}
\usepackage{blindtext}
\usepackage{cases}
\usepackage{empheq}
\usepackage{multicol}
\usepackage{placeins} % Allows float barriers for spacing

\usepackage{wrapfig}
\usepackage{caption}

\usepackage{subfigure}
\usepackage{float}

\usepackage{color}

\usepackage[color=white,linecolor=black]{todonotes}
\usepackage{mathrsfs}
\usepackage{fontenc} %T1 font encoding
\usepackage{inputenc} %UTF-8 support
\usepackage{enumitem}
\usepackage{verbatim}
\numberwithin{equation}{section}
\theoremstyle{plain}
\newtheorem{theorem}{Theorem}[section]
\newtheorem{proposition}[theorem]{Proposition}

\theoremstyle{definition}
\newtheorem{definition}[theorem]{Definition}

\theoremstyle{remark}
\newtheorem{remark}[theorem]{Remark}

\theoremstyle{remark}
\newtheorem*{note}{Note}
\usepackage[square,comma,numbers,sort]{natbib}
\usepackage[colorlinks=true, pdfborder={ 0 0 0}]{hyperref}
\hypersetup{urlcolor=blue, citecolor=red}
\usepackage{url}
\makeatletter
\renewcommand{\section}{\@startsection{section}{1}{0pt}%
  {1.5ex plus .2ex minus .2ex}%
  {1.0ex plus .2ex}%
  {\normalfont\large\bfseries\raggedright}}
\renewcommand{\subsection}{\@startsection{subsection}{2}{0pt}%
  {1.25ex plus .2ex minus .2ex}%
  {1.0ex plus .2ex}%
  {\normalfont\normalsize\bfseries\raggedright}}
\makeatother

\allowdisplaybreaks
\makeatletter
\def\abstractname{Abstract}
\renewenvironment{abstract}{%
  \begin{center}
    \textbf{\abstractname}
  \end{center}
}{}
\makeatother

\usepackage{xcolor}
\makeatletter
\def\@settitle{%
  \begin{center}%
    \normalfont\Large\bfseries
    \@title\par
  \end{center}%
}
\makeatother

\title[Weak solutions and higher regularity]{
 Existence of Weak Solutions and Higher-Order Regularity for a Heat-Wave Fluid-Structure Interaction System on a Periodic Strip
}

 \author[M. M. Rahman]{Mohammad Mahabubur Rahman}
\address[Mohammad Mahabubur Rahman]{School of Mathematical and Statistical Sciences, Clemson University, Clemson, SC 29634, USA}
\email[Mohammad Mahabubur Rahman]{rahman6@clemson.edu}
\begin{document}
\maketitle
\begin{abstract}
In this paper, we study a heat-wave fluid-structure interaction system
posed on a periodic strip geometry and investigate whether higher-order
$H^s$-estimates for arbitrarily large values of $s$ can be established for the
coupled system despite the inherent mismatch of parabolic and hyperbolic regularity.  We establish global existence of
weak solutions and Sobolev regularity up to the $H^3$-level, which, to
the best of our knowledge, is the highest regularity currently established
for this system. We further show that, even for arbitrarily regular initial
data, neither the semigroup approach through higher-order generator domains
$D(A^k)$ with $k\in\mathbb{N}$ nor the PDE-based tangential-normal recovery procedure yields
closed estimates beyond $H^3$ due to the complexities of the coupled PDE
system.

\end{abstract}
\section{Introduction}

\noindent Fluid-structure interaction (FSI) systems arise in a broad range of applications including hemodynamics, aeroelasticity, biomechanics, microfluidics, geophysical flows, and engineering systems involving compliant structures interacting with fluids. Mathematically, these systems couple fluid and structural equations of different types, most commonly parabolic and hyperbolic equations, through interface trace and normal coupling conditions. The interaction across the interface creates analytical difficulties that are absent in the uncoupled equations because the fluid and structure possess fundamentally different regularity mechanisms. In moving-boundary models, these challenges are further compounded by evolving geometries, nonlinear ALE transformations, free-boundary compactness issues, geometric nonlinearities, and possible self-contact phenomena.

\medskip

\noindent The mathematical theory of fluid-structure interaction has undergone substantial development over the past two decades. Foundational results on the existence of weak solutions for nonlinear fluid-structure interaction systems were established in
\cite{barbu2007existence,MuhaCanic2013},
and were subsequently extended to more general shell and three-dimensional structural models in
\cite{MuhaSchwarzacher2022, bociu2022weak, CanicMuhaTawri2024,BrandtCanicMuha2025}. Related developments for fluid-rigid body interactions, compressible models, viscoelastic systems, and contact problems can be found in
\cite{chemetov2019weak, MuhaNecasovaRadosevic2023, bociu2022weak,KampschulteMuhaTrifunovic2023,BukalKukavicaLiMuha2025}.

\medskip

\noindent Beyond the existence of weak solutions, strong solvability and higher-order regularity for coupled fluid-structure systems have also received considerable attention. In particular, representative contributions are given by
\cite{mitra2018local, kukavica2009strong,Lequeurre2011},
where strong solutions and higher regularity were established for several classes of coupled fluid-structure interaction problems.

\medskip

\noindent Alongside these developments,  a closely related direction concerns fixed-domain fluid-structure interaction systems, where the principal analytical difficulty is no longer the motion of the interface but rather the coupling of parabolic and hyperbolic equations through interface trace and normal coupling conditions. This class includes coupled heat-wave systems and Stokes-Lam\'e models, which provide fundamental prototypes for studying parabolic-hyperbolic interactions. Foundational contributions include
\cite{AvalosTriggiani2007,AvalosLasieckaTriggiani2008,AvalosTriggiani2009,ZhangZuazua2007}. These works established semigroup well-posedness, characterized the associated generators, and investigated qualitative properties such as long-time behavior for coupled heat-wave and Stokes-Lam\'e systems.

\medskip

\noindent Subsequent developments have focused primarily on stabilization, decay estimates, observability, interface control,  $H^2$-estimates, and regularizing effects for coupled parabolic-hyperbolic and and fluid-structure interaction  systems
\cite{AvalosLasieckaTriggiani2008, AlbanoTataru2000,Duyckaerts2006,Zuazua2001,avalos2015rational,avalos2016heat,peralta2018interface}. In particular, Muha \cite{muha2015note} established optimal Sobolev regularity together with regularizing effects for a one-dimensional coupled heat-wave system. These contributions provide a comprehensive understanding of the well-posedness and qualitative behavior of coupled parabolic-hyperbolic systems. However, they do not establish higher-order Sobolev regularity for multidimensional coupled heat-wave fluid-structure interaction systems.
\medskip

\noindent Having reviewed the recent developments in fluid-structure interaction, we now consider the coupled heat-wave system studied in the present work.

\begin{subequations}\label{1.0}
\begin{align}
u_t-\Delta u&=0
&&\text{in }(0,T)\times\Omega_f,
\label{1.0a}
\\
w_{tt}-\Delta w+w&=0
&&\text{in }(0,T)\times\Omega_s,
\label{1.0b}
\\
u&=0
&&\text{on }(0,T)\times\Gamma_f,
\label{1.0c}
\\
w&=0
&&\text{on }(0,T)\times\Gamma_{s,0},
\label{1.0d}
\\
u&=w_t
&&\text{on }(0,T)\times\Gamma_s,
\label{1.0e}
\\
\partial_\nu u&=\partial_\nu w
&&\text{on }(0,T)\times\Gamma_s,
\label{1.0f}
\\
u(0)&=u_0
&&\text{in }\Omega_f,
\nonumber
\\
w(0)&=w_0,
\qquad
w_t(0)=w_1
&&\text{in }\Omega_s.
\label{1.0g}
\end{align}
\end{subequations}
In \eqref{1.0}, the unknown
$
u=[u_1(t,x), u_2(t,x)]
$
denotes the heat component in the fluid region $\Omega_f$, while
$
w=[w_1(t,x), w_2(t,x)]
$
denotes the structural displacement in the structure region $\Omega_s$. The
interaction between the two components occurs through the coupling conditions
\eqref{1.0e}-\eqref{1.0f}. We  also note that all Sobolev and Lebesgue spaces pertaining to $u$, $w$, and $w_t$
are in fact $(H^s)^n$ and $(L^2)^n$, respectively, with $n=2$, but we
omit the exponent $n$ for the sake of simplicity.

\noindent \textbf{Geometry for the heat-wave fluid- structure interaction system \eqref{1.0}.}
\noindent Let $(x_1,x_2)\in\mathbb T_L^1\times(0,h_2)$ and define the structure and fluid subdomains by
$
\Omega_s := \mathbb T_L^1\times(0,h_1),\;
\Omega_f := \mathbb T_L^1\times(h_1,h_2),
$
with flat interface
$
\Gamma_s := \mathbb T_L^1\times\{x_2=h_1\},
$
and outer fluid boundary
$
\Gamma_f := \mathbb T_L^1\times\{x_2=h_2\}.
$
\begin{center}
\includegraphics[scale=0.60]{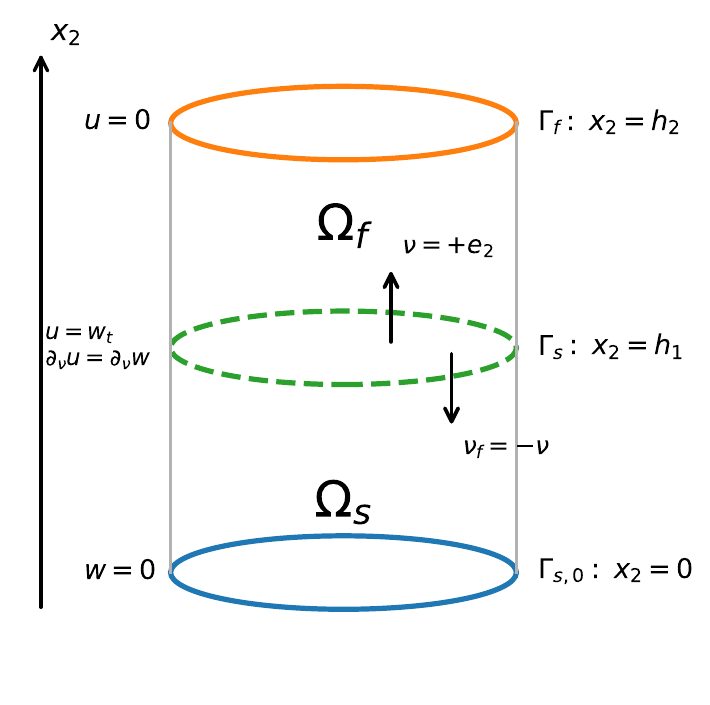}

\textbf{Figure: Geometry of the FSI Domain}
\end{center}
 The lower boundary of the structure is
$
\Gamma_{s,0} := \mathbb T_L^1\times\{x_2=0\},
$
so that
$
\partial\Omega_s=\Gamma_{s,0}\cup\Gamma_s,
\;
\partial\Omega_f=\Gamma_s\cup\Gamma_f.
$ The unit normal on $\Gamma_s$ is chosen as
$\nu=e_2,$
pointing from $\Omega_s$ into $\Omega_f$. Consequently, the outward unit normal of $\Omega_f$ on $\Gamma_s$ is $-\nu$.
We impose the homogeneous boundary conditions
$$
u=0 \quad \text{on } (0,T)\times\Gamma_f,
\qquad
w=0 \quad \text{on } (0,T)\times\Gamma_{s,0}.
$$
Consequently, boundary integrals arising from integration by parts vanish on
$\Gamma_f$ and $\Gamma_{s,0}$, and  contributions on $\Gamma_s$
remain.
\medskip

\noindent The model is posed on a periodic strip geometry with a flat interface. Although this setting eliminates many of the geometric difficulties encountered in moving-boundary fluid-structure interaction problems, the central analytical difficulty remains. The two interface conditions couple the heat equation \eqref{1.0a} and the wave equation \eqref{1.0b}, so that higher-order $H^s$-estimates require the regularity of the fluid and structure to be established simultaneously. Consequently, every increase in Sobolev order requires higher-order spatial derivatives of $u$, higher-order spatial and temporal derivatives of $w$, and repeated differentiation of the interface conditions \eqref{1.0e}-\eqref{1.0f}.
\medskip

\noindent Having described the present geometry and the structure of the coupled heat-wave PDE  system\eqref{1.0}, we record the following technical observation.

\begin{remark}
In the present flat geometry
$
\Gamma_s=\mathbb T_L^1\times\{x_2=h_1\},
$
the unit normal and tangential vectors are constant. Consequently,
tangential differentiation does not generate geometric commutator terms
associated with curvature, variable normals, or coordinate changes.

\medskip 

\noindent  More precisely, the tangential differential operator
$
D^\alpha:=\partial_{x_1}^\alpha
$
commutes exactly with both the Laplacian and the normal derivative. Since
$
\Delta=\partial_{x_1}^2+\partial_{x_2}^2
$
and
$
\partial_\nu=\pm\partial_{x_2},
$
we have
$$
[D^\alpha,\Delta]=0,
\qquad
[D^\alpha,\partial_\nu]=0,
$$
for every integer $\alpha \in \mathbb{N}_0$.

\noindent Consequently, tangential differentiation of \eqref{1.0} produces no commutator terms either in the interior equations or in the trace and normal derivative relations on $\Gamma_s$. In particular,
$$
D^\alpha u
=
D^\alpha w_t,
\qquad
\partial_\nu D^\alpha u
=
\partial_\nu D^\alpha w
\quad
\text{on }\Gamma_s.
$$
\end{remark}
\noindent   We now turn to the main regularity question addressed in this paper:
\textit{Can  higher-order $H^s$-estimates be established for the  heat-wave fluid-structure interaction system despite the inherent mismatch between the regularity scales of the parabolic and hyperbolic components?} We show that the answer is affirmative up to the $H^3$-level, which, to the best of our knowledge, is the highest Sobolev regularity currently established for this coupled system. Furthermore, we show that, even under arbitrarily high regularity assumptions on the initial data, neither the semigroup approach based on higher-order generator domains $D(A^k)$ nor the PDE-based tangential-normal recovery procedure yields closed estimates beyond the $H^3$-level.
\medskip

\noindent \textbf{Plan of the paper.} The remainder of the paper is organized as follows. Section~2 introduces the functional spaces, notation, and the definition of the generator together with its domain.  Section~3 investigates
the structural obstructions to establishing  $H^s$-estimates
for the coupled heat-wave fluid-structure interaction system by analyzing the
recursive tangential-normal recovery procedure and the higher-order
generator-domain approach. Section~4 presents the main results together with
their proofs. In particular, we establish the global existence of weak
solutions through the Galerkin approximation, derive the corresponding uniform
energy estimates, carry out the passage to the limit, establish the time-derivative estimates and recover the weak formulation, prove weak continuity
of the energy variable together with the attainment of the prescribed initial
data, and finally establish the higher-order regularity of the coupled system
up to the $H^3$-level. 
\medskip

\noindent Throughout the paper, we use the following spaces and notation unless otherwise specified.
\section{Spaces and notation}
\noindent We define
$$
V_f
:=
\{
v\in H^1(\Omega_f):
v=0 \text{ on }\Gamma_f
\},
\qquad
V_s
:=
\left\{
v\in H^1(\Omega_s):
v=0 \text{ on }\Gamma_{s,0}
\right\}.
$$
To incorporate the trace coupling condition on 
$\Gamma_s$, we introduce the coupled space
$$
\mathcal V
:=\{
(\varphi_f, \varphi_s)\in V_f\times V_s:
\varphi_f\rvert_{\Gamma_s}
=
\varphi_s\rvert_{\Gamma_s}
\}.
$$
The space $\mathcal V$ is equipped with the norm
$$
\lVert(\varphi_f, \varphi_s)\rVert_{\mathcal V}^2
:=
\lVert\varphi_f\rVert_{H^1(\Omega_f)}^2
+
\lVert\varphi_s\rVert_{H^1(\Omega_s)}^2.
$$
Its dual space is denoted by $\mathcal V'$.
The natural energy space for the heat-wave fluid-structure interaction system \eqref{1.0} is
$$
\mathbb H
:=
L^2(\Omega_f)
\times
V_s
\times
L^2(\Omega_s),
$$
endowed with the norm
$$
\lVert(u,w,z)\rVert_{\mathbb H}^2
:=
\lVert u\rVert_{L^2(\Omega_f)}^2
+
\lVert w\rVert_{H^1(\Omega_s)}^2
+
\lVert z\rVert_{L^2(\Omega_s)}^2.$$
In addition, we introduce
$$
\mathcal H
:=
L^2(\Omega_f)\times L^2(\Omega_s),
\qquad
\big((a,b),(c,d)\big)_{\mathcal H}
:=
(a,c)_{L^2(\Omega_f)}
+
(b,d)_{L^2(\Omega_s)}.
$$

\noindent For convenience, we write
$
X(t)
:=
\bigl(u(t),w(t),w_t(t)\bigr)
$
for the energy variable associated with the coupled system.
 Throughout the paper, we use
$
D^\alpha
:=
\partial_{x_1}^{\alpha},
\;
\alpha\in\mathbb N_0,
$
to denote tangential derivatives in the periodic direction $x_1$. In particular,
$D^\alpha u
=\partial_{x_1}^{\alpha}u$ and when mixed derivatives are required, we write
$$
\partial_{x_1}^{\alpha}\partial_{x_2}^{b}u,
\qquad
\alpha,b\in\mathbb N_0,
$$
We also recall that the  Sobolev norm involves derivatives in both
the tangential and normal directions:
$$
\lVert u\rVert_{H^s(\Omega)}^2
=
\sum_{\substack{\alpha,b\ge0\\\alpha+b \le s}}
\left\lVert
\partial_{x_1}^{\alpha}\partial_{x_2}^{b}u
\right\rVert_{L^2(\Omega)}^2.
$$
 We next introduce the abstract evolution formulation associated with the  heat-wave fluid-structure interaction system \eqref{1.0}. In particular, we define the corresponding generator and its domain.

\begin{definition}[Generator and its domain]
For the heat-wave fluid-structure interaction system \eqref{1.0}, set
$
v:=w_t, \textrm{and}\;
X:=(u,w,v).
$
We define the space
$
\mathbb H
:=
L^2(\Omega_f)
\times
V_s
\times
L^2(\Omega_s),
$
equipped with the norm
$
\lVert (u,w,v)\rVert_{\mathbb H}^2
:=
\lVert u\rVert_{L^2(\Omega_f)}^2
+
\lVert w\rVert_{H^1(\Omega_s)}^2
+
\lVert v\rVert_{L^2(\Omega_s)}^2$. The operator $A:D(A)\subset\mathbb H\to\mathbb H$ is defined by
$$
A
\begin{pmatrix}
u\\
w\\
v
\end{pmatrix}
=
\begin{pmatrix}
\Delta u\\
v\\
\Delta w-w
\end{pmatrix},
$$
or equivalently,
$$
A
=
\begin{pmatrix}
\Delta & 0 & 0\\
0 & 0 & I\\
0 & \Delta-I & 0
\end{pmatrix}.
$$
Thus the system \eqref{1.0} can be written formally as
$
X_t=AX.
$ Since
$$
AX=(\Delta u,\;v,\;\Delta w-w),
$$
we require
$$
\Delta u\in L^2(\Omega_f),
\quad
\Delta w-w\in L^2(\Omega_s),
\quad
v\in V_s,
$$
together with the boundary conditions \eqref{1.0c}-\eqref{1.0d} and coupling conditions \eqref{1.0e}-\eqref{1.0f} prescribed in \eqref{1.0}. Thus, the domain of $A$ is defined by 
\begin{align}
D(A)
=
\left\{
(u,w,v)\in \mathbb H,
\begin{array}{l}
u\in H^2(\Omega_f),\\
w\in H^2(\Omega_s),\\
v\in V_s,\\
u=0 \text{ on }\Gamma_f,\\
w=0 \text{ on }\Gamma_{s,0},\\
u=v \text{ on }\Gamma_s,\\
\partial_{\nu}u=\partial_{\nu}w \text{ on }\Gamma_s
\end{array}
\right\}.
\label{DA}
\end{align}
Here the condition $u=v$ on $\Gamma_s$ is the trace coupling condition,
since $v=w_t$, while
$
\partial_\nu u=\partial_\nu w
\;
\text{on }\Gamma_s
$
is the normal coupling condition.
\end{definition} 
\section{Limitations of higher-order Sobolev regularity methods}
\noindent We  investigate the possibility of establishing higher-order $H^s$-estimates for the heat-wave fluid-structure interaction system \eqref{1.0}. The following analysis identifies a structural bottleneck that prevents the extension of the regularity theory beyond the $H^3$-level. Rather than presenting heuristic difficulties, we examine two standard analytical frameworks for higher-order Sobolev regularity and identify the points at which the corresponding estimates cease to close. Specifically, we consider the recursive tangential-normal recovery procedure and the higher-order generator-domain approach.
\medskip

\noindent The difficulties in obtaining complete isotropic $H^s$-regularity for the coupled heat-wave system can be understood from the following two viewpoints:
\begin{itemize}
\item[(i)] the recursive tangential-normal recovery procedure,
\item[(ii)] the higher-order generator-domains $D(A^k)$.
\end{itemize}

\subsection{Limitation of the recursive tangential-normal recovery procedure}

\noindent We now explain precisely what is obtained from the available tangential energy estimates and why these estimates do not yield a full isotropic $H^s$-estimates for the fluid variable $u$.

\noindent Since
$
\Delta=\partial_{x_1}^2+\partial_{x_2}^2,
$
 \eqref{1.0a}
can be rewritten as
\begin{align}
\partial_{x_2}^2u
=
u_t-\partial_{x_1}^2u.
\label{nru}
\end{align}
From the tangential energy estimate \eqref{BHs}, we have
\begin{align}
\sup_{t\in[0,T]}
\sum_{\alpha=0}^{s}
\lVert \partial_{x_1}^{\alpha}u(t)\rVert_{L^2(\Omega_f)}^2
\le C,
\label{1.4R}
\end{align}
together with
\begin{align}
\int_0^T
\sum_{\alpha=0}^{s}
\lVert
\nabla\partial_{x_1}^{\alpha}u(t)
\rVert_{L^2(\Omega_f)}^2\,dt
\le C.
\label{1.5R}
\end{align}
We therefore deduce from \eqref{1.4R} and  \eqref{1.5R} that
\begin{align}
\label{Tan-u}
\partial_{x_1}^{\alpha}u
\in
L^\infty(0,T;L^2(\Omega_f)),
\;
0\le\alpha\le s,
\end{align}
and
$
\nabla\partial_{x_1}^{\alpha}u
\in
L^2(0,T;L^2(\Omega_f)),\; 
0\le\alpha\le s
$.
\noindent In particular, we obtain one normal derivative in the time-integrated sense,
$$
\partial_{x_1}^{\alpha}\partial_{x_2}u
\in
L^2(0,T;L^2(\Omega_f)),
\qquad
0\le\alpha\le s.
$$
Similarly, \eqref{LemmaA} yields
\begin{align}
\sup_{t\in[0,T]}
\sum_{\alpha=0}^{s-1}
\lVert \partial_{x_1}^{\alpha}u_t(t)\rVert_{L^2(\Omega_f)}^2
\le C,
\label{Tan}
\end{align}
and
\begin{align}
\int_0^T
\sum_{\alpha=0}^{s-1}
\lVert
\nabla\partial_{x_1}^{\alpha}u_t(t)
\rVert_{L^2(\Omega_f)}^2\,dt
\le C.
\label{aut-g}
\end{align}
Applying $\partial_{x_1}^k$ to \eqref{nru}, for $0\le k\le s-2$, gives
\begin{align}
\partial_{x_2}^2\partial_{x_1}^k u
=
\partial_{x_1}^k u_t
-
\partial_{x_1}^{k+2}u.
\label{nru-k}
\end{align}
Using \eqref{Tan} and
\eqref{Tan-u}, we deduce
\begin{align}
\sup_{t\in[0,T]}
\sum_{k=0}^{s-2}
\lVert
\partial_{x_2}^2\partial_{x_1}^k u(t)
\rVert_{L^2(\Omega_f)}^2
\le C.
\label{Secub}
\end{align}
However, this estimate alone does not imply the full isotropic bound
$
u\in L^\infty(0,T;H^s(\Omega_f))$.
Indeed, the full Sobolev norm requires control of
\begin{align}
\lVert u(t)\rVert_{H^s(\Omega_f)}^2
=
\sum_{\substack{\alpha,b\ge0\\\alpha+b \le s}}
\lVert
\partial_{x_1}^{\alpha}\partial_{x_2}^b u(t)
\rVert_{L^2(\Omega_f)}^2.
\label{FullHs}
\end{align}
The estimate \eqref{Secub} controls only the derivatives with exactly two normal derivatives,
$
\partial_{x_2}^2\partial_{x_1}^k u,
\;
0\le k\le s-2,
$
but it does not control higher normal derivatives such as
$
\partial_{x_1}\partial_{x_2}^3u,
\;
\partial_{x_2}^4u.$
To recover these missing terms, let $\alpha,b\ge0$, $b\ge2$, and $\alpha+b \le s$. Applying
$
\partial_{x_1}^{\alpha}\partial_{x_2}^{b-2}
$
to \eqref{nru} yields
\begin{align}
\partial_{x_1}^{\alpha}\partial_{x_2}^{b}u
=
\partial_{x_1}^{\alpha}\partial_{x_2}^{b-2}u_t
-
\partial_{x_1}^{\alpha+2}\partial_{x_2}^{b-2}u,
\label{nrg}
\end{align}
and hence implies
\begin{align}
\lVert
\partial_{x_1}^{\alpha}\partial_{x_2}^{b}u(t)
\rVert_{L^2(\Omega_f)}
&\le
\lVert
\partial_{x_1}^{\alpha}\partial_{x_2}^{b-2}u_t(t)
\rVert_{L^2(\Omega_f)}
+
\lVert
\partial_{x_1}^{\alpha+2}\partial_{x_2}^{b-2}u(t)
\rVert_{L^2(\Omega_f)},
\label{nrg-es}
\end{align}
where the second term contains two fewer normal derivatives of $u$, and therefore can be treated inductively once lower-order normal derivatives are known. The first term in \eqref{nrg-es}, however, requires the estimate
\begin{align}
u_t
\in
L^\infty(0,T;H^{s-2}(\Omega_f)).
\label{mH}
\end{align}
This estimate in \eqref{mH} is not supplied by the tangential energy method. The available bounds provide only tangential control of $u_t$:
$
\partial_{x_1}^{\alpha}u_t
\in
L^\infty(0,T;L^2(\Omega_f)),\;
0\le\alpha\le s-1,$ that is, $u_t
\in
L^\infty(0,T;H_{x_1}^{\alpha}(\Omega_f)) )$,
but do not provide all mixed derivatives
$
\partial_{x_1}^{\alpha}\partial_{x_2}^b u_t,\;
\alpha+b \le s-2
$. To recover the missing normal derivatives of $u_t$, differentiate \eqref{1.0a} with respect to time. Since
$
u_{tt}-\Delta u_t=0,
$
and
$
\Delta u_t
=
\partial_{x_1}^2u_t+\partial_{x_2}^2u_t,
$
we obtain
\begin{align}
\partial_{x_2}^2u_t
=
u_{tt}-\partial_{x_1}^2u_t.
\label{1.15F}
\end{align}
Applying the same recovery argument to $u_t$ in \eqref{1.15F} now requires
\begin{align}
u_{tt}
\in
L^\infty(0,T;H^{s-4}(\Omega_f)).
\label{mu_tt}
\end{align}
Repeating this procedure recursively yields
\begin{align}
u\in H^s(\Omega_f)
&\Longleftarrow
u_t\in H^{s-2}(\Omega_f),
\\
u_t\in H^{s-2}(\Omega_f)
&\Longleftarrow
u_{tt}\in H^{s-4}(\Omega_f),
\\
u_{tt}\in H^{s-4}(\Omega_f)
&\Longleftarrow
u_{ttt}\in H^{s-6}(\Omega_f),
\end{align}
and so forth.
Consequently, the recursive tangential-normal recovery procedure generates an entire sequence of required higher time-derivative estimates. In order to prove
$
u\in H^s(\Omega_f),
$
one must first establish
$
u_t\in H^{s-2}(\Omega_f).
$
However, proving
$
u_t\in H^{s-2}(\Omega_f)
$
requires
$
u_{tt}\in H^{s-4}(\Omega_f),
$
and proving this estimate further requires
$
u_{ttt}\in H^{s-6}(\Omega_f).
$
Thus each recovery step requires a new higher isotropic Sobolev estimate for a higher time derivative, and the proof becomes circular. Consequently, the tangential-normal recovery procedure fails to produce full $H^s$-estimates for the  heat-wave fluid-structure interaction system \eqref{1.0}.
\medskip

\noindent Thus the estimates rigorously obtained from the present argument remain anisotropic. In particular, we obtain
\begin{align}
\partial_{x_1}^{\alpha}u
&\in
L^\infty(0,T;L^2(\Omega_f)),
\qquad
0\le\alpha\le s,
\\
\nabla\partial_{x_1}^{\alpha}u
&\in
L^2(0,T;L^2(\Omega_f)),
\qquad
0\le\alpha\le s,
\\
\partial_{x_2}^2\partial_{x_1}^{k}u
&\in
L^\infty(0,T;L^2(\Omega_f)),
\qquad
0\le k\le s-2,
\\
\partial_{x_1}^{\alpha}u_t
&\in
L^\infty(0,T;L^2(\Omega_f)),
\qquad
0\le\alpha\le s-1,
\\
\nabla\partial_{x_1}^{\alpha}u_t
&\in
L^2(0,T;L^2(\Omega_f)),
\qquad
0\le\alpha\le s-1.
\end{align}
Equivalently, these estimates may be seen in the following
\begin{align}
u
&\in
L^\infty(0,T;H^s_{x_1}L^2_{x_2}(\Omega_f)),
\\
u
&\in
L^2(0,T;H^{s+1}_{x_1}L^2_{x_2}(\Omega_f))
\cap
L^2(0,T;H^s_{x_1}H^1_{x_2}(\Omega_f)),
\\
u
&\in
L^\infty(0,T;H^{s-2}_{x_1}H^2_{x_2}(\Omega_f)),
\\
u_t
&\in
L^\infty(0,T;H^{s-1}_{x_1}L^2_{x_2}(\Omega_f)),
\\
u_t
&\in
L^2(0,T;H^s_{x_1}L^2_{x_2}(\Omega_f))
\cap
L^2(0,T;H^{s-1}_{x_1}H^1_{x_2}(\Omega_f)).
\end{align}
\medskip

\noindent
An analogous obstruction occurs for the structure variable $w$.  Therefore, without additional regularity assumptions, these estimates do not imply
$$
u\in L^\infty(0,T;H^s(\Omega_f)),
\qquad
w\in L^\infty(0,T;H^s(\Omega_s)),
\qquad
w_t\in L^\infty(0,T;H^{s-1}(\Omega_s)).
$$
The discussion above highlights the difficulties in extending the estimates to $H^s$-regularity. In particular, the following note shows how these difficulties manifest themselves at the $H^4$-level, which is the first level beyond the established $H^3$-estimates.
\begin{note}
In particular, the obstruction appears when one attempts to continue the
normal recovery procedure beyond the $H^3$-level. At the $H^3$-level,
\eqref{thnle} shows that $\partial_{x_2}^3 w$ can be controlled by quantities
already bounded by the preceding estimates. Consequently, the normal recovery
argument closes at the $H^3$-level. However, one further differentiation in \eqref{thnl} yields 
\begin{align*}
\partial_{x_2}^4 w
=
\partial_{x_2}^2 w_{tt}
-
\partial_{x_1}^2\partial_{x_2}^2 w
+
\partial_{x_2}^2 w,
\end{align*}
 where the term
$$
\partial_{x_1}^2\partial_{x_2}^2 w
$$
require additional derivatives that are not supplied by the classical
second-order diffusion and \eqref{BHs}. Consequently, the derivative budget provided by the
Laplacian is exhausted at the $H^3$-level, and the normal recovery procedure
cannot be continued to obtain $H^4$-estimates. This identifies the
fundamental bottleneck in the regularity  of the coupled system and
explains why the $H^3$-level cannot be iterated to obtain
$H^4$-estimates.
\end{note}

\subsection{Limitations of the higher-order generator-domains  $D(A^k)$}

\noindent
We now explain why the semigroup framework based on the higher generator domains $D(A^k)$ does not by itself yield  $H^s$-estimates for the heat-wave fluid-structure interaction system \eqref{1.0}.

\noindent Let
$
Y=(u,w,v),
\;
v=w_t,
$
and recall from Definition \ref{DA} that
$
A(u,w,v)
=
(\Delta u,v,\Delta w-w).
$
Then \eqref{1.0} may be written abstractly as
$
Y_t=AY$. We recall the recursive definition of the iterated domains (see,
\cite[Definition 2.3.1]{zheng2004nonlinear}):
$$
D(A^k)
=
\{u\mid u\in D(A^{k-1}),\ Au\in D(A^{k-1})\},
\qquad k\in\mathbb N.
$$
For example,
$$
Y\in D(A^2)
\quad\Longleftrightarrow\quad
Y\in D(A)
\quad\text{and}\quad
AY\in D(A).
$$
Since
$
AY
=
(\Delta u,v,\Delta w-w),
$
the condition
$
AY\in D(A)
$
means that the triple
$
(\Delta u,v,\Delta w-w)
$
must satisfy the defining conditions in \eqref{DA}. Since $u=v$ on $\Gamma_s$ is part of the definition \ref{DA}, applying \eqref{DA} to the triple
$
(\Delta u,v,\Delta w-w)
$
gives
\begin{align}
\Delta u
=
\Delta w-w
\qquad
\text{on }\Gamma_s.
\label{DA2W}
\end{align}
Similarly, since
$
\partial_\nu u
=
\partial_\nu w
\;
\text{on }\Gamma_s
$
is also part of \eqref{DA}, applying \eqref{DA} to
$
(\Delta u,v,\Delta w-w)
$
gives
\begin{align}
\partial_\nu\Delta u
=
\partial_\nu(\Delta w-w)
\qquad
\text{on }\Gamma_s.
\label{DA2PO}
\end{align}
At first sight, one might try to recover
$
u\in H^4(\Omega_f)
$
from the regularity of $\Delta u$. However, elliptic recovery at the $H^4$-level requires both sufficiently regular interior forcing and sufficiently regular boundary traces. More precisely, by the classical elliptic regularity theory for the Dirichlet Laplacian (see Lions-Magenes \cite[p.~VI, Preface; Chapter~2]{LionsMagenes1972}), one has
\begin{align}
\lVert u\rVert_{H^4(\Omega_f)}
\le
C\Big(
\lVert \Delta u\rVert_{H^2(\Omega_f)}
+
\lVert u\rvert_{\partial\Omega_f}\rVert_{H^{7/2}(\partial\Omega_f)}
\Big).
\label{u-H4}
\end{align}
From \eqref{1.0a},
we obtain $
\Delta u=u_t.
$
Hence recovering
$
u\in H^4(\Omega_f)
$
through \eqref{u-H4} requires both
$
\Delta u\in H^2(\Omega_f)
$
and
$
u\rvert_{\partial\Omega_f}
\in
H^{7/2}(\partial\Omega_f)$. From
$
Y\in D(A^2),
$
we have
$
AY=(\Delta u,v,\Delta w-w)\in D(A).
$
Therefore, by \eqref{DA},
$$
\Delta u\in H^2(\Omega_f),
\qquad
v\in H^2(\Omega_s),
\qquad
\Delta w-w\in V_s,
$$
together with the relations
\eqref{DA2W}-\eqref{DA2PO} on $\Gamma_s$. Moreover, since
$
u=v
\;
\text{on }\Gamma_s,
$
and
$
v\in H^2(\Omega_s),
$
the trace theorem yields
$$
u\rvert_{\Gamma_s}
\in
H^{3/2}(\Gamma_s).
$$
On the other hand, the elliptic estimate \eqref{u-H4} requires the stronger boundary regularity
$$
u\rvert_{\partial\Omega_f}
\in
H^{7/2}(\partial\Omega_f).
$$
In particular, this requires
$$
u\rvert_{\Gamma_s}
\in
H^{7/2}(\Gamma_s).
$$
However, the information provided by
$
Y\in D(A^2)
$
yields only
$$
u\rvert_{\Gamma_s}
\in
H^{3/2}(\Gamma_s),
$$
and
$$
H^{3/2}(\Gamma_s)
\not\subset
H^{7/2}(\Gamma_s),
$$
hence, the regularity supplied by $D(A^2)$ is insufficient to verify the boundary regularity required by \eqref{u-H4}, and thus does not yield
$$
u\in H^4(\Omega_f).
$$
A similar obstruction appears for the structure component. Recovering
$
w\in H^4(\Omega_s)
$
would require an estimate of the form
\begin{align}
\lVert w\rVert_{H^4(\Omega_s)}
\le
C\Big(
\lVert \Delta w\rVert_{H^2(\Omega_s)}
+
\lVert w\rVert_{\Gamma_{s,0}}\rVert_{H^{7/2}(\Gamma_{s,0})}
+
\lVert \partial_\nu w\rvert_{\Gamma_s}\rVert_{H^{5/2}(\Gamma_s)}
\Big).
\label{w-H4}
\end{align}
From \eqref{1.0b},
we obtain
\begin{align}
\Delta w
=
w_{tt}+w,
\label{w-lap}
\end{align}
and recovering
$
w\in H^4(\Omega_s)
$
through \eqref{w-H4} requires
\begin{align}
w_{tt}+w
\in
H^2(\Omega_s).
\label{wtt-H2}
\end{align}
From \eqref{1.0d}, it follows that
$
w=0
\;
\text{on }\Gamma_{s,0},
$
so $\lVert w\rvert_{\Gamma_{s,0}}\rVert_{H^{7/2}(\Gamma_{s,0})}=0$, as seen in \eqref{w-H4}. 
Hence, if
$
u\in H^2(\Omega_f),
$
then the trace theorem gives
\begin{align}
\partial_\nu u\rvert_{\Gamma_s}
\in
H^{1/2}(\Gamma_s).
\label{HLTL}
\end{align}
Since
$
\partial_\nu w
=
\partial_\nu u
\;
\text{on }\Gamma_s,$ and
\eqref{HLTL} yields only
\begin{align}
\partial_\nu w\rvert_{\Gamma_s}
\in
H^{1/2}(\Gamma_s),
\label{67TA1}
\end{align}
however, \eqref{w-H4} requires,
\begin{align}
\partial_\nu w\rvert_{\Gamma_s}
\in
H^{5/2}(\Gamma_s),
\end{align} therefore \eqref{67TA1} is insufficient to close \eqref{w-H4}.
Furthermore,
$
Y\in D(A^2)
$
implies only that
$$
\Delta w-w\in V_s,
$$ and the elliptic estimate \eqref{w-H4} requires the stronger interior regularity requirement
\eqref{wtt-H2}. Therefore, the information obtained from
$
Y\in D(A^2)
$
is insufficient to close the elliptic estimate \eqref{w-H4}, and consequently does not by itself imply
$
w\in H^4(\Omega_s).
$

\noindent In contrast, semigroup theory implies that if
$
Y_0\in D(A^k),
$
then the corresponding solution satisfies
$$
Y\in \bigcap_{j=0}^k C^{k-j}([0,+\infty);D(A^j)),
$$
see, \cite[Theorem~2.4]{zheng2004nonlinear}. However, this result yields
regularity only in the iterated generator domains $D(A^j)$. To obtain
isotropic Sobolev regularity, one must identify these domains with the
corresponding Sobolev spaces. The preceding analysis shows that, for the
 heat-wave fluid-structure interaction system\eqref{1.0}, the information contained in $D(A^2)$ is
insufficient to recover
$$
u\in H^4(\Omega_f),
\qquad
w\in H^4(\Omega_s).
$$
Hence, higher-order generator-domain regularity does not by itself
yield higher-order Sobolev regularity for the coupled system.

\medskip

\noindent
Therefore, the semigroup-generator iteration based on the higher domains $D(A^k)$ does not by itself close the elliptic estimates needed to recover higher-order Sobolev regularity for the  heat-wave fluid-structure interaction system \eqref{1.0}.

\medskip

\noindent We now introduce the corresponding notion of weak solution for the heat-wave fluid-structure interaction system \eqref{1.0}.
\begin{definition}[Weak solution]
\label{esw}
A triple
$
X=(u,w,w_t)
$
is called a weak solution of \eqref{1.0} on $[0,T]$ if the following conditions hold:
\begin{enumerate}

\item \label{Item 1}
$
X\in L^\infty(0,T;\mathbb H),
\;
u\in L^2(0,T;H^1(\Omega_f)).
$

\item
$
\partial_t(u,w_t)
\in
L^2(0,T;\mathcal V').
$
\item
For every $(\varphi_f, \varphi_s)\in\mathcal V$ and for a.e. $t\in(0,T)$,
\begin{align}
\langle
\partial_t(u,w_t)(t),
(\varphi_f,\varphi_s)
\rangle_{\mathcal V',\mathcal V}
+
(\nabla u(t),\nabla\varphi_f)_{\Omega_f}
+
(\nabla w(t),\nabla\varphi_s)_{\Omega_s}
+
(w(t),\varphi_s)_{\Omega_s}
=
0.
\label{1.33H}
\end{align}
\end{enumerate}
The trace and normal coupling conditions are understood through the coupled
variational formulation \eqref{1.33H} and the choice of the test space
$\mathcal V$.
\end{definition}
\begin{remark}
The formulation of weak solutions in Definition~\ref{esw} differs from
that employed by
Barbu, Gruji\'c, Lasiecka, and Tuffaha \cite{barbu2007existence}. Their analysis
first considers semigroup solutions corresponding to initial data in
the generator domain $D(A_L)$. At this level of regularity, the
boundary traces satisfy
$$
w_t\rvert_{\Gamma_s}=u\rvert_{\Gamma_s}
\in H^{1/2}(\Gamma_s),
\qquad
\sigma(w)\cdot \nu\in H^{-1/2}(\Gamma_s).
$$
The extension to energy-level solutions is then obtained by density
arguments together with a hidden regularity result for the hyperbolic
component, yielding
$
\sigma(w) \cdot \nu
\in
L^2(0,T;H^{-1/2}(\Gamma_s)).
$
\medskip

\noindent
In contrast, the present analysis is carried out entirely in the
natural energy space
$$
\mathbb H
=
L^2(\Omega_f)\times V_s\times L^2(\Omega_s),
$$
without imposing additional regularity assumptions on the initial
data. Although
$
u\in L^2(0,T;H^1(\Omega_f))
$
implies
$$
u\rvert_{\Gamma_s}
\in
L^2(0,T;H^{1/2}(\Gamma_s)),
$$
the energy regularity does not provide the traces
$$
w_t\rvert_{\Gamma_s},
\qquad
\partial_\nu u\rvert_{\Gamma_s},
\qquad
\partial_\nu w\rvert_{\Gamma_s}.
$$
Consequently, the trace and normal coupling conditions
$$
u=w_t,
\qquad
\partial_\nu u=\partial_\nu w,
$$
are incorporated through the coupled variational formulation
\eqref{1.33H} and the choice of the space $\mathcal V$.
\end{remark}
\section{Main Results}
\noindent 
In this section, we present the main results of the paper. We first establish the global existence and uniqueness of weak solutions together with the corresponding trace  coupling condition in Theorem \ref{Th 2.1}. We then derive the higher-order tangential and time-differentiated a priori estimates in Proposition \ref{Prop 2.2}. Finally, using these estimates together with the coupled equations, we obtain the $H^3$-estimates for the heat-wave fluid-structure interaction system\eqref{1.0} in Theorem \ref{Th 2.3}.

\begin{theorem}
\label{Th 2.1} Let $\mathbb H= L^2(\Omega_f)\times V_s\times L^2(\Omega_s)
)$.
Assume that
$X(0)=
(u_0,w_0,w_1)
\in \mathbb{H}$.  
Then for every $T>0$, there exists a weak solution
$
X(t):=(u(t),w(t),w_t(t)) \in L^{\infty}(0,T;\mathbb{H})
$
to \eqref{1.0} in the sense of Definition \ref{esw}.Moreover,
$$
X\in C_w([0,T];\mathbb H),
$$
and the solution fulfills the initial condition in the sense
$
\lim_{t\to0^+}\lVert X(t)-X_0\rVert_{\mathbb H}=0.
$
In addition, the trace coupling relation $u\rvert_{\Gamma_s}=w_t\rvert_{\Gamma_s}$ is recovered in the Galerkin trace sense and  uniqueness holds within the class of weak solutions obtained through the Galerkin approximation.
\end{theorem}
\noindent The uniqueness statement in Theorem~\ref{Th 2.1}
is restricted to weak solutions obtained through the Galerkin
approximation. The following remark explains why uniqueness for
arbitrary weak solutions in the energy space remains open.
\begin{remark}
Let
$
\widehat X
=
(\widehat u,\widehat w,\widehat z)
$
be the difference of two weak solutions satisfying Definition~\ref{esw}. From
the regularity obtained in the existence theorem \ref{Th 2.1}, we have
$$
\widehat u\in L^2(0,T;V_f),
\qquad
\widehat w\in L^\infty(0,T;V_s),
\qquad
\widehat z=\widehat w_t
\in L^\infty(0,T;L^2(\Omega_s)).
$$
Subtracting \eqref{1.33H} for the two weak solutions yields the corresponding
difference equation. To derive the associated energy identity, one would
naturally test this difference equation with
$$
(\varphi_f,\varphi_s)
=
(\widehat u,\widehat z),
$$
which requires
$$
\widehat z\in L^2(0,T;V_s),
$$
so that
$$
(\widehat u,\widehat z)\in L^2(0,T;\mathcal V).
$$
However, the regularity obtained in the existence theorem does not provide
$$
\widehat z\in L^2(0,T;V_s).
$$
This loss of regularity originates from the hyperbolic component of the
coupled system \eqref{1.0}, for which the existence of weak solution result yields only
$$
w_t\in L^\infty(0,T;L^2(\Omega_s)).
$$
However, if the initial data are assumed to belong to $D(A)$, then the
semigroup regularity theorem \cite[Chapter~1, Theorem~2.4]{pazy2012semigroups} yields
$
X(t)\in D(A)
\;
\text{for all }t\in[0,T]$. Consequently,
$
w_t(t)\in V_s$,
and therefore, for the difference of two such regular solutions, we have
$$
\widehat z=\widehat w_t\in L^2(0,T;V_s),$$ and 
hence
$$
(\widehat u,\widehat z)\in L^2(0,T;\mathcal V).$$
Therefore, one can establish uniqueness by testing the difference equation
with
$
(\widehat u,\widehat z)$; see, for example, \cite[Corollary~3.1, Theorem~4.1]{barbu2007existence}. Such a result, however, requires
 stronger regularity assumptions than the regularity available in
Definition~\ref{esw}. Therefore, the analysis establishes uniqueness only within the class of
weak solutions obtained through the Galerkin approximation. This argument
does not yield uniqueness for arbitrary weak solutions in the energy space,
because the regularity available in Definition~\ref{esw} is insufficient to
justify testing the difference equation with $(\widehat u,\widehat z)$. The question of uniqueness for arbitrary weak solutions in the
energy space
$
\mathbb H
=
L^2(\Omega_f)\times V_s\times L^2(\Omega_s)
$
remains open.
\end{remark}
\noindent
\noindent
With this clarification concerning the scope of the uniqueness statement, we now proceed to the proof of Theorem~\ref{Th 2.1}. 
\begin{proof}[Proof of Theorem \ref{Th 2.1}]
We proceed as follows (i) Galerkin approximation-formulation, (ii) solvability of Galerkin approximation and energy estimates, (iii) passage to the limit, (iv)  energy inequality for the limit, (v) time derivative estimate and recovery of the weak formulation, (vi) weak continuity of the energy space, 
(vii) initial condition, 
(viii) Strong continuity at the initial time, (ix) uniqueness
\medskip

\noindent\underline{\textbf{Galerkin approximation-formulation:}}
We introduce
$$
V_f
=
\{v\in H^1(\Omega_f):v=0\text{ on }\Gamma_f\},
\qquad
V_s
=
\{v\in H^1(\Omega_s):v=0\text{ on }\Gamma_{s,0}\},
$$
and define the coupled space
$$
\mathcal V
=
\{
(\varphi_f,\varphi_s)\in V_f\times V_s:
\varphi_f\rvert_{\Gamma_s}
=
\varphi_s\rvert_{\Gamma_s}
\}.
$$
For the coupled pair $(u,w_t)$, we introduce
\begin{align}
\mathcal H
:=
L^2(\Omega_f)\times L^2(\Omega_s),
\qquad
\big((a,b),(c,d)\big)_{\mathcal H}
:=
(a,c)_{L^2(\Omega_f)}
+
(b,d)_{L^2(\Omega_s)}.
\label{Hspace}
\end{align}
The space $\mathcal V$ is a closed subspace of $V_f\times V_s$. Since
$V_f$ and $V_s$ are separable Hilbert spaces, $\mathcal V$ is also a
separable Hilbert space. Hence, there exists a countable orthonormal basis of $\mathcal V$,
$$
\{(\Phi_k,\Psi_k)\}_{k\ge1}\subset\mathcal V,
$$
with respect to the $\mathcal V$-inner product. In particular, this basis
is complete in $\mathcal V$, and hence
\begin{align}
\overline{
\operatorname{span}
\{(\Phi_k,\Psi_k):k\ge1\}
}^{\,\mathcal V}
=
\mathcal V.
\label{V-dense}
\end{align}
For each $N\in\mathbb N$, define
\begin{align}
\mathcal V_N
:=
\operatorname{span}
\{(\Phi_1,\Psi_1),\dots,(\Phi_N,\Psi_N)\}
\subset\mathcal V,
\label{VN}
\end{align}
and the completeness of the basis yields
\begin{align}
\overline{
\bigcup_{N=1}^{\infty}\mathcal V_N
}^{\,\mathcal V}
=
\mathcal V.
\label{VN-dense}
\end{align}
Since $(\Phi_k,\Psi_k)\in\mathcal V$, each basis pair satisfies
\begin{align}
\Phi_k\rvert_{\Gamma_f}
&=0,
&
\Psi_k\rvert_{\Gamma_{s,0}}
&=0,
&
\Phi_k\rvert_{\Gamma_s}
&=
\Psi_k\rvert_{\Gamma_s}.
\label{bby}
\end{align}
We choose approximate initial data
\begin{align}
(u_0^N,w_1^N)
&\in
\mathcal V_N,
&
w_0^N
&\in
\operatorname{span}\{\Psi_1,\dots,\Psi_N\},
\end{align}
such that
\begin{align}
u_0^N
\to u_0
\quad\text{in }L^2(\Omega_f),
\qquad
w_0^N
\to w_0
\quad\text{in }H^1(\Omega_s),
\qquad
w_1^N
\to w_1
\quad\text{in }L^2(\Omega_s).
\label{OPE56}
\end{align}
Since $(u_0^N,w_1^N)\in\mathcal V_N$, there exist coefficients
$g_1(0),\dots,g_N(0)$ such that
\begin{align}
u_0^N
=
\sum_{k=1}^{N} g_k(0)\Phi_k,
\qquad
w_1^N
=
\sum_{k=1}^{N} g_k(0)\Psi_k.
\label{AQmZ6}
\end{align}
In addition, there exist coefficients $h_1(0),\dots,h_N(0)$ such that
\begin{align}
w_0^N
=
\sum_{k=1}^{N} h_k(0)\Psi_k.
\label{w0N-exp}
\end{align}
We seek approximate solutions of the form
\begin{align}
u^N(t)
&=
\sum_{k=1}^{N} g_k(t)\Phi_k,
\label{unaz}
\\
w^N(t)
&=
\sum_{k=1}^{N} h_k(t)\Psi_k,
\label{EROL}
\end{align}
and impose
\begin{align}
\dot h_k(t)=g_k(t),
\qquad
k=1,\dots,N.
\label{ghrelation}
\end{align}
From 
\eqref{AQmZ6}, \eqref{w0N-exp} and \eqref{unaz}, \eqref{EROL}, we obtain
\begin{align}
u^N(0)=u_0^N,
\qquad
w^N(0)=w_0^N.
\label{1L}
\end{align}
Moreover, since $\dot h_k(t)=g_k(t)$, from \eqref{EROL}, we have
\begin{align}
w_t^N(0)
=
\sum_{k=1}^{N}\dot h_k(0)\Psi_k
=
\sum_{k=1}^{N}g_k(0)\Psi_k
=
w_1^N.
\label{1KO}
\end{align}
Combine \eqref{1L} and \eqref{1KO}  implies
\begin{align}
u^N(0)=u_0^N,
\qquad
w^N(0)=w_0^N,
\qquad
w_t^N(0)=w_1^N.
\label{2.10PG}
\end{align}
We then proceed with \eqref{EROL} to obtain
\begin{align}
w_t^N(t)
=
\sum_{k=1}^{N} \dot h_k(t)\Psi_k
=
\sum_{k=1}^{N} g_k(t)\Psi_k.
\label{wtanz}
\end{align}
Using \eqref{bby}, \eqref{unaz}, and \eqref{wtanz}, we obtain
\begin{align}
u^N\rvert_{\Gamma_s}
=
w_t^N\rvert_{\Gamma_s},
\label{WGRq}
\end{align}
thus the trace coupling condition is satisfied at the Galerkin level. Moreover, since
$
\Psi_k\rvert_{\Gamma_{s,0}}=0,
\;
k=1,\dots,N,
$
it follows from \eqref{EROL} that
$
w^N\rvert_{\Gamma_{s,0}}
=
0.$
The normal coupling condition
$$
\partial_\nu u^N
=
\partial_\nu w^N
\qquad
\text{on }\Gamma_s,
$$
is incorporated through the coupled variational formulation. Indeed, for every
$
(\varphi_f,\varphi_s)\in\mathcal V_N,
$
testing \eqref{1.0a} by $\varphi_f$ and \eqref{1.0b} by
$\varphi_s$, integrating by parts in space, using the boundary conditions
on $\Gamma_f$ and $\Gamma_{s,0}$, the trace identity
$
\varphi_f\rvert_{\Gamma_s}
=
\varphi_s\rvert_{\Gamma_s},
$
and the normal coupling on $\Gamma_s$, yields
\begin{align}
&(u_t^N,\varphi_f)_{\Omega_f}
+
(\nabla u^N,\nabla\varphi_f)_{\Omega_f}
+
(w_{tt}^N,\varphi_s)_{\Omega_s}
+
(\nabla w^N,\nabla\varphi_s)_{\Omega_s}
+
(w^N,\varphi_s)_{\Omega_s}
=
0,
\label{Gw}
\end{align}
for every
$
(\varphi_f,\varphi_s)\in\mathcal V_N.
$
\medskip

\noindent\underline{\textbf{Solvability of Galerkin approximation and energy estimates:}}
Using \eqref{unaz}, \eqref{EROL}, \eqref{wtanz} and taking $(\varphi_f,\varphi_s)=(\Phi_i,\Psi_i)$, $i=1,\dots,N$, in
\eqref{Gw}  gives
\begin{align}
\sum_{j=1}^{N}M_{ij}\dot g_j(t)
+
\sum_{j=1}^{N}A_{ij}g_j(t)
+
\sum_{j=1}^{N}B_{ij}h_j(t)
&=0,
\qquad i=1,\dots,N,
\label{Galerkin-ODE1}
\\
\dot h_i(t)&=g_i(t),
\qquad i=1,\dots,N,
\label{Galerkin-ODE2}
\end{align}
where
\begin{align}
M_{ij}
&=
(\Phi_j,\Phi_i)_{L^2(\Omega_f)}
+
(\Psi_j,\Psi_i)_{L^2(\Omega_s)},
\\
A_{ij}
&=
(\nabla\Phi_j,\nabla\Phi_i)_{\Omega_f},
\\
B_{ij}
&=
(\nabla\Psi_j,\nabla\Psi_i)_{\Omega_s}
+
(\Psi_j,\Psi_i)_{\Omega_s}.
\end{align}
The mass matrix $M=(M_{ij})_{i,j=1}^{N}$ is symmetric and positive
definite. Indeed, for any $\alpha=(\alpha_1,\dots,\alpha_N)\in\mathbb R^N$,
\begin{align}
\sum_{i,j=1}^{N}M_{ij}\alpha_i\alpha_j
=
\left\lVert 
\sum_{j=1}^{N}\alpha_j\Phi_j
\right\rVert_{L^2(\Omega_f)}^2
+
\left\lVert 
\sum_{j=1}^{N}\alpha_j\Psi_j
\right\rVert_{L^2(\Omega_s)}^2.
\end{align}

\noindent Choosing $(\varphi_f,\varphi_s)=(\Phi_m,\Psi_m)$ in \eqref{Gw}, multiplying
by $g_m(t)$, summing over $m=1,\dots,N$, and using
\eqref{unaz} and \eqref{wtanz}, we obtain
\begin{align}
&(u_t^N,u^N)_{\Omega_f}
+
\lVert \nabla u^N\rVert_{L^2(\Omega_f)}^2
+
(w_{tt}^N,w_t^N)_{\Omega_s}
+
(\nabla w^N,\nabla w_t^N)_{\Omega_s}
+
(w^N,w_t^N)_{\Omega_s}
=
0,
\label{ensp}
\end{align}
 where \eqref{ensp} implies 
\begin{align}
\frac{d}{dt}\mathcal{E}_N(t)
+
\lVert \nabla u^N(t)\rVert_{L^2(\Omega_f)}^2
=
0,
\label{EnID}
\end{align}
where
\begin{align}
\mathcal{E}_N(t)
:=
\frac12\lVert u^N(t)\rVert_{L^2(\Omega_f)}^2
+
\frac12\lVert w_t^N(t)\rVert_{L^2(\Omega_s)}^2
+
\frac12\lVert \nabla w^N(t)\rVert_{L^2(\Omega_s)}^2
+
\frac12\lVert w^N(t)\rVert_{L^2(\Omega_s)}^2.
\label{enfl}
\end{align}
Integrating \eqref{EnID} over $(0,t)$, we get
\begin{align}
\mathcal{E}_N(t)
+
\int_0^t
\lVert \nabla u^N(\tau)\rVert_{L^2(\Omega_f)}^2\,d\tau
=
\mathcal{E}_N(0).
\label{eqEB}
\end{align}
\underline{\textbf{Passage to the limit:}}
The energy identity \eqref{eqEB} yields the uniform bounds
\begin{align}
u^N \text{ in }
L^\infty(0,T;L^2(\Omega_f))
\cap
L^2(0,T;H^1(\Omega_f)),\;
w^N \text{ in }
L^\infty(0,T;H^1(\Omega_s)),\;
w_t^N \text{ in }
L^\infty(0,T;L^2(\Omega_s)).
\label{2.29S1}
\end{align}
 By the Banach-Alaoglu theorem, from \eqref{2.29S1}, there exists
$
(u,w,w_t)$
and a subsequence (still indexed by $N$) such that
\begin{align}
u^N \overset{*}{\rightharpoonup} u
&\quad\text{in } L^\infty(0,T;L^2(\Omega_f)), \label{coul}\\
u^N \rightharpoonup u
&\quad\text{in } L^2(0,T;H^1(\Omega_f)), \label{45PQn}\\
\nabla u^N \rightharpoonup \nabla u
&\quad\text{in } L^2(0,T;L^2(\Omega_f)), \label{2Aqzp}\\
w^N \overset{*}{\rightharpoonup} w
&\quad\text{in } L^\infty(0,T;H^1(\Omega_s)), \label{OTep2}\\
\nabla w^N \overset{*}{\rightharpoonup} \nabla w
&\quad\text{in } L^\infty(0,T;L^2(\Omega_s)), \label{ONW2}\\
w_t^N \overset{*}{\rightharpoonup} w_t
&\quad\text{in } L^\infty(0,T;L^2(\Omega_s)), \label{conwl2}
\end{align}
Fix $M\in\mathbb{N}$, let
$
(\varphi_f,\varphi_s)\in\mathcal V_M,
$
where
$
\mathcal V_M
:=
\operatorname{span}
\{(\Phi_1,\Psi_1),\dots,(\Phi_M,\Psi_M)\},
$ and let $\psi\in C^\infty([0,T])$ satisfy $\psi(T)=0$. Since $\mathcal {V}_M\subset \mathcal{V}_N$ for every $N\ge M$.
Multiplying \eqref{1.0a} by $\varphi_f\psi(t)$ and \eqref{1.0b} by
$\varphi_s\psi(t)$, integrating over $\Omega_f\times(0,T)$ and
$\Omega_s\times(0,T)$, respectively, and using
$\partial\Omega_f=\Gamma_f\cup\Gamma_s$, $\partial\Omega_s=\Gamma_{s,0}\cup\Gamma_s$,
$\varphi_f=0$ on $\Gamma_f$, $\varphi_s=0$ on $\Gamma_{s,0}$, and the normal
coupling cancellation on $\Gamma_s$, we obtain
\begin{align}
\int_0^T
(u_t^N,\varphi_f)_{\Omega_f}\psi(t)\,dt
&+
\int_0^T
(\nabla u^N,\nabla\varphi_f)_{\Omega_f}\psi(t)\,dt+
\int_0^T
(w_{tt}^N,\varphi_s)_{\Omega_s}\psi(t)\,dt
\nonumber\\
&+
\int_0^T
(\nabla w^N,\nabla\varphi_s)_{\Omega_s}\psi(t)\,dt
+
\int_0^T
(w^N,\varphi_s)_{\Omega_s}\psi(t)\,dt
=
0.
\label{bIBP}
\end{align}
We now integrate the time derivatives by parts. Since
$\psi(T)=0$, we obtain
\begin{align}
\int_0^T
(u_t^N,\varphi_f)_{\Omega_f}\psi(t)\,dt
&=
-\int_0^T
(u^N,\varphi_f)_{\Omega_f}\psi'(t)\,dt
-
(u_0^N,\varphi_f)_{\Omega_f}\psi(0),
\label{IBP-u}
\end{align}
and
\begin{align}
\int_0^T
(w_{tt}^N,\varphi_s)_{\Omega_s}\psi(t)\,dt
&=
-\int_0^T
(w_t^N,\varphi_s)_{\Omega_s}\psi'(t)\,dt
-
(w_1^N,\varphi_s)_{\Omega_s}\psi(0).
\label{IBP-w}
\end{align}
Substituting \eqref{IBP-u} and \eqref{IBP-w} into
\eqref{bIBP} yields
\begin{align}
0
&=
-\int_0^T
(u^N(t),\varphi_f)_{\Omega_f}\psi'(t)\,dt
-
(u_0^N,\varphi_f)_{\Omega_f}\psi(0)
\nonumber\\
&\quad
+
\int_0^T
(\nabla u^N(t),\nabla\varphi_f)_{\Omega_f}\psi(t)\,dt
\nonumber\\
&\quad
-
\int_0^T
(w_t^N(t),\varphi_s)_{\Omega_s}\psi'(t)\,dt
-
(w_1^N,\varphi_s)_{\Omega_s}\psi(0)
\nonumber\\
&\quad
+
\int_0^T
(\nabla w^N(t),\nabla\varphi_s)_{\Omega_s}\psi(t)\,dt
+
\int_0^T
(w^N(t),\varphi_s)_{\Omega_s}\psi(t)\,dt .
\label{GIN}
\end{align}
Using
\eqref{coul}-\eqref{conwl2}, \eqref{OPE56}, and passing to the limit $N\to\infty$ in
\eqref{GIN}, we obtain
\begin{align}
0
&=
-\int_0^T
(u(t),\varphi_f)_{\Omega_f}\psi'(t)\,dt
-
(u_0,\varphi_f)_{\Omega_f}\psi(0)
\nonumber\\
&\quad
+
\int_0^T
(\nabla u(t),\nabla\varphi_f)_{\Omega_f}\psi(t)\,dt
\nonumber\\
&\quad
-
\int_0^T
(w_t(t),\varphi_s)_{\Omega_s}\psi'(t)\,dt
-
(w_1,\varphi_s)_{\Omega_s}\psi(0)
\nonumber\\
&\quad
+
\int_0^T
(\nabla w(t),\nabla\varphi_s)_{\Omega_s}\psi(t)\,dt
+
\int_0^T
(w(t),\varphi_s)_{\Omega_s}\psi(t)\,dt .
\label{lire}
\end{align} 
Since $M$ was arbitrary, \eqref{lire} holds for every
$
(\varphi_f,\varphi_s)\in\bigcup_{M=1}^{\infty}\mathcal V_M.
$
By \eqref{VN-dense}, for any $(\varphi_f,\varphi_s)\in\mathcal V$, there
exists a sequence
$$
(\varphi_f^m,\varphi_s^m)\in\bigcup_{M=1}^{\infty}\mathcal V_M
$$
such that
$$
(\varphi_f^m,\varphi_s^m)\to(\varphi_f,\varphi_s)
\qquad
\text{in }\mathcal V.
$$
Passing to the limit $m\to\infty$ in \eqref{lire} with
$(\varphi_f^m,\varphi_s^m)$ in place of
$(\varphi_f,\varphi_s)$, we obtain \eqref{lire} for every
$$
(\varphi_f,\varphi_s)\in\mathcal V.
$$
\medskip

\noindent \noindent
We next recover the interface coupling conditions. From \eqref{45PQn}, we obtain
$
u\in L^2(0,T;H^1(\Omega_f)),
$
hence, by the Sobolev trace theorem,
\begin{align}
u\rvert_{\Gamma_s}
\in
L^2(0,T;H^{1/2}(\Gamma_s)).
\label{2.71GH5}
\end{align}
By \eqref{45PQn} and the continuity of the trace operator, we obtain
\begin{align}
u^N\rvert_{\Gamma_s}
\rightharpoonup
u\rvert_{\Gamma_s}
\qquad
\text{weakly in }
L^2(0,T;H^{1/2}(\Gamma_s)).
\label{TL1}
\end{align}
Using \eqref{WGRq} and \eqref{TL1}, we define the trace of $w_t$ on $\Gamma_s$ by
\begin{align}
w_t\rvert_{\Gamma_s}
:=
u\rvert_{\Gamma_s}
\qquad
\text{in }
L^2(0,T;H^{1/2}(\Gamma_s)),
\end{align}
which implies that
\begin{align}
w_t\rvert_{\Gamma_s}
=
u\rvert_{\Gamma_s}
\end{align}
holds in the Galerkin trace sense.

\medskip

\noindent
The variational identity \eqref{1.33H} is obtained from the Galerkin formulation after cancellation of the boundary coupling terms and therefore incorporates the normal coupling condition in the weak sense.
\medskip

\noindent \underline{\textbf{Time derivative estimate and recovery of the weak formulation:}} For every
$
(\varphi_f,\varphi_s)\in\mathcal V_N,
$ \eqref{Gw} gives
\begin{align}
(u_t^N,\varphi_f)_{\Omega_f}
+
(w_{tt}^N,\varphi_s)_{\Omega_s}
&=
-
(\nabla u^N,\nabla\varphi_f)_{\Omega_f}
-
(\nabla w^N,\nabla\varphi_s)_{\Omega_s}
-
(w^N,\varphi_s)_{\Omega_s},
\label{det-VN}
\end{align}
from which it follows that
\begin{align}
\rvert
(u_t^N,\varphi_f)_{\Omega_f}
+
(w_{tt}^N,\varphi_s)_{\Omega_s}
\rvert
&\le
\lVert \nabla u^N\rVert_{L^2(\Omega_f)}
\lVert \nabla\varphi_f\rVert_{L^2(\Omega_f)}
\nonumber\\
&\quad
+
\lVert \nabla w^N\rVert_{L^2(\Omega_s)}
\lVert \nabla\varphi_s\rVert_{L^2(\Omega_s)}
\nonumber\\
&\quad
+
\lVert w^N\rVert_{L^2(\Omega_s)}
\lVert \varphi_s\rVert_{L^2(\Omega_s)}
\nonumber\\
&\le
C
\left(
\lVert u^N\rVert_{H^1(\Omega_f)}
+
\lVert w^N\rVert_{H^1(\Omega_s)}
\right)
\lVert(\varphi_f,\varphi_s)\rVert_{\mathcal V},
\label{Due2-VN}
\end{align}
taking the supremum over
$
0\neq(\varphi_f,\varphi_s)\in\mathcal V_N
$
with respect to the $\mathcal V$-norm, we obtain
\begin{align}
\lVert (u_t^N,w_{tt}^N)\rVert_{\mathcal V_N'}
&\le
C
\left(
\lVert u^N\rVert_{H^1(\Omega_f)}
+
\lVert w^N\rVert_{H^1(\Omega_s)}
\right),
\label{DB-VN}
\end{align}
where $\mathcal V_N'$ denotes the dual of $ \mathcal V_N$ equipped with the norm induced by $\mathcal V$. Therefore, by \eqref{2.29S1},
\begin{align}
\int_0^T
\lVert (u_t^N(t),w_{tt}^N(t))\rVert_{\mathcal V_N'}^2\,dt
\le
C.
\label{Df-VN}
\end{align}
Since the spaces $\mathcal V_N'$ depend on $N$, \eqref{Df-VN} is only a Galerkin-level estimate. To obtain compactness in a fixed dual space, we introduce the corresponding residual functional on $\mathcal V$.

\medskip

\noindent
For each $t\in(0,T)$, define $F_N(t)\in\mathcal V'$ by
\begin{align}
\langle
F_N(t),(\varphi_f,\varphi_s)
\rangle_{\mathcal V',\mathcal V}
:=
-
(\nabla u^N(t),\nabla\varphi_f)_{\Omega_f}
-
(\nabla w^N(t),\nabla\varphi_s)_{\Omega_s}
-
(w^N(t),\varphi_s)_{\Omega_s}
\label{FN-def}
\end{align}
for every
$
(\varphi_f,\varphi_s)\in\mathcal V,
$
which yields
\begin{align}
\lvert
\left\langle
F_N(t),(\varphi_f,\varphi_s)
\right\rangle_{\mathcal V',\mathcal V}
\rvert
&\le
C
\left(
\lVert u^N(t)\rVert_{H^1(\Omega_f)}
+
\lVert w^N(t)\rVert_{H^1(\Omega_s)}
\right)
\lVert(\varphi_f,\varphi_s)\rVert_{\mathcal V},
\end{align}
and, taking the supremum over all
$
(\varphi_f,\varphi_s)\in\mathcal V
$
with
$
\lVert(\varphi_f,\varphi_s)\rVert_{\mathcal V}\le1,
$
gives
\begin{align}
\lVert F_N(t)\rVert_{\mathcal V'}
&\le
C
\left(
\lVert u^N(t)\rVert_{H^1(\Omega_f)}
+
\lVert w^N(t)\rVert_{H^1(\Omega_s)}
\right).
\label{REWQj}
\end{align}
Squaring \eqref{REWQj}, integrating over $(0,T)$, and using \eqref{2.29S1}, from \eqref{FN-def}, we obtain
\begin{align}
\int_0^T
\lVert F_N(t)\rVert_{\mathcal V'}^2\,dt
\le
C,
\label{4.48HK}
\end{align}
which implies that $\{F_N\}_{N\ge1}$ is bounded in $L^2(0,T;\mathcal V')$. Since $L^2(0,T;\mathcal V')$ is a Hilbert space, the Banach-Alaoglu theorem gives a subsequence, still denoted by $N$, and an element
$
G\in L^2(0,T;\mathcal V')
$
such that
\begin{align}
F_N
\rightharpoonup
G
\qquad
\text{weakly in }L^2(0,T;\mathcal V').
\label{G-limit}
\end{align}
We now identify $G$ as the weak time derivative of
$
Z:=(u,w_t).$
Fix
$
(\varphi_f,\varphi_s)\in\mathcal V$
and $
\eta\in C_c^\infty(0,T).$
By the density of
$
\bigcup_{N=1}^{\infty}\mathcal V_N$
in $\mathcal V$, there exists
$
(\varphi_f^M,\varphi_s^M)\in\mathcal V_M
$
such that
\begin{align}
(\varphi_f^M,\varphi_s^M)
\to
(\varphi_f,\varphi_s)
\qquad
\text{strongly in }\mathcal V
\qquad
\text{as }M\to\infty.
\label{2.59M}
\end{align}
For fixed $M$, take $N\ge M$. Then for
$
(\varphi_f^M,\varphi_s^M)\in\mathcal V_N,
$ using \eqref{FN-def},  \eqref{det-VN}, we obtain
\begin{align}
(u_t^N,\varphi_f^M)_{\Omega_f}
+
(w_{tt}^N,\varphi_s^M)_{\Omega_s}
= \langle
F_N,(\varphi_f^M,\varphi_s^M)
\rangle_{\mathcal V',\mathcal V}.
\label{det-M}
\end{align}
Multiplying \eqref{det-M} by $\eta(t)$, integrating over $(0,T)$, and integrating by parts in time, we get
\begin{align}
-\int_0^T
\Big[
(u^N(t),\varphi_f^M)_{\Omega_f}
+
(w_t^N(t),\varphi_s^M)_{\Omega_s}
\Big]\eta'(t)\,dt
&=
\int_0^T
\left\langle
F_N(t),(\varphi_f^M,\varphi_s^M)
\right\rangle_{\mathcal V',\mathcal V}
\eta(t)\,dt.
\label{G1-M}
\end{align}
 Since $M$ is fixed, the test pair in \eqref{G1-M} is independent of $N$. Passing to the limit $N\to\infty$, using \eqref{coul}, \eqref{conwl2}, and \eqref{G-limit}, gives
\begin{align}
-\int_0^T
\Big[
(u(t),\varphi_f^M)_{\Omega_f}
+
(w_t(t),\varphi_s^M)_{\Omega_s}
\Big]\eta'(t)\,dt
&=
\int_0^T
\left\langle
G(t),(\varphi_f^M,\varphi_s^M)
\right\rangle_{\mathcal V',\mathcal V}
\eta(t)\,dt.
\label{G2-M}
\end{align}
Letting $M\to\infty$ in \eqref{G2-M}, using \eqref{2.59M}, yields
\begin{align}
-\int_0^T
\Big[
(u(t),\varphi_f)_{\Omega_f}
+
(w_t(t),\varphi_s)_{\Omega_s}
\Big]\eta'(t)\,dt
&=
\int_0^T
\left\langle
G(t),(\varphi_f,\varphi_s)
\right\rangle_{\mathcal V',\mathcal V}
\eta(t)\,dt.
\label{G2}
\end{align}
By the definition of weak derivative, \eqref{G2}
implies that $G$ is the weak time derivative of $Z=(u,w_t)$ in $\mathcal V'$. Hence
\begin{align}
\partial_t(u,w_t)
=
G
\qquad
\text{in }L^2(0,T;\mathcal V').
\label{Zt-G}
\end{align}
It remains to identify the equation satisfied by $G$. Since
$
F_N\rightharpoonup G
$
weakly in
$
L^2(0,T;\mathcal V')
$ from \eqref{4.48HK}
and
$
\eta(\varphi_f,\varphi_s)\in L^2(0,T;\mathcal V),
$
we have
\begin{align}
\int_0^T
\langle
F_N(t),\eta(t)(\varphi_f,\varphi_s)
\rangle_{\mathcal V',\mathcal V}
\,dt
\to
\int_0^T
\left\langle
G(t),\eta(t)(\varphi_f,\varphi_s)
\right\rangle_{\mathcal V',\mathcal V}
\,dt.
\label{2ub2.65}
\end{align}
On the other hand, by ,
\begin{align}
\int_0^T
\langle
F_N(t),\eta(t)(\varphi_f,\varphi_s)
\rangle_{\mathcal V',\mathcal V}
\,dt
&=
-\int_0^T
(\nabla u^N(t),\nabla\varphi_f)_{\Omega_f}\eta(t)\,dt
\nonumber\\
&\quad
-\int_0^T
(\nabla w^N(t),\nabla\varphi_s)_{\Omega_s}\eta(t)\,dt
-
\int_0^T
(w^N(t),\varphi_s)_{\Omega_s}\eta(t)\,dt.
\label{FN2.66}
\end{align}
Combining 
\eqref{2ub2.65} with \eqref{FN2.66}, we obtain
\begin{align}
\int_0^T
\left\langle
G(t),(\varphi_f,\varphi_s)
\right\rangle_{\mathcal V',\mathcal V}
\eta(t)\,dt
&=
-\int_0^T
(\nabla u(t),\nabla\varphi_f)_{\Omega_f}\eta(t)\,dt
\nonumber\\
&\quad
-\int_0^T
(\nabla w(t),\nabla\varphi_s)_{\Omega_s}\eta(t)\,dt
-
\int_0^T
(w(t),\varphi_s)_{\Omega_s}\eta(t)\,dt.
\label{Eq2.67H}
\end{align} 
Using \eqref{Zt-G} in \eqref{Eq2.67H}, we obtain
\begin{align}
\int_0^T
\Big[
\langle
\partial_t(u,w_t)(t),
(\varphi_f,\varphi_s)
\rangle_{\mathcal V',\mathcal V}
+
(\nabla u(t),\nabla\varphi_f)_{\Omega_f}
+
(\nabla w(t),\nabla\varphi_s)_{\Omega_s}
+
(w(t),\varphi_s)_{\Omega_s}
\Big]\eta(t)\,dt
=
0.
\label{BCF3}
\end{align}
The bracketed expression in \eqref{BCF3} belongs to $L^1(0,T)$. Indeed,
\begin{align}
\lvert 
\langle
\partial_t(u,w_t)(t),
(\varphi_f,\varphi_s)
\rangle_{\mathcal V',\mathcal V}
\rvert
&\le
\lVert \partial_t(u,w_t)(t)\rVert_{\mathcal V'}
\lVert(\varphi_f,\varphi_s)\rVert_{\mathcal V},
\end{align}
and the remaining terms are integrable because
$
u\in L^2(0,T;H^1(\Omega_f))
$
and
$
w\in L^\infty(0,T;H^1(\Omega_s)).
$
Since
$
\eta\in C_c^\infty(0,T)
$
is arbitrary, the fundamental lemma of the calculus of variations gives
\begin{align}
\langle
\partial_t(u,w_t)(t),
(\varphi_f,\varphi_s)
\rangle_{\mathcal V',\mathcal V}
+
(\nabla u(t),\nabla\varphi_f)_{\Omega_f}
+
(\nabla w(t),\nabla\varphi_s)_{\Omega_s}
+
(w(t),\varphi_s)_{\Omega_s}
=
0,
\label{WNCW}
\end{align}
for a.e. $t\in(0,T)$ and every fixed
$
(\varphi_f,\varphi_s)\in\mathcal V.
$
\medskip

\noindent \underline{\textbf{Energy inequality for the limit:}}
 Let
$
X^N(t):=(u^N(t),w^N(t),w_t^N(t))
$
and
$
X_0^N:=(u_0^N,w_0^N,w_1^N),
$
then
$$
\lVert X^N(t)\rVert_{\mathbb H}^2
=
\lVert u^N(t)\rVert_{L^2(\Omega_f)}^2
+
\lVert w^N(t)\rVert_{H^1(\Omega_s)}^2
+
\lVert w_t^N(t)\rVert_{L^2(\Omega_s)}^2.
$$
It follows from  \eqref{eqEB} that, for every $t\in[0,T]$,
\begin{align}
\frac12\lVert X^N(t)\rVert_{\mathbb H}^2
+
\int_0^t
\lVert \nabla u^N(\tau)\rVert_{L^2(\Omega_f)}^2\,d\tau
=
\frac12\lVert X_0^N\rVert_{\mathbb H}^2.
\label{2.38FSIW}
\end{align}
Let
$\psi\in C_c^\infty(0,T),
\;
\psi\ge0.
$
Multiplying \eqref{2.38FSIW} by $\psi(t)$ and integrating over $(0,T)$, we obtain
\begin{align}
\int_0^T
\Big[
\frac12\lVert X^N(t)\rVert_{\mathbb H}^2
+
\int_0^t
\lVert \nabla u^N(\tau)\rVert_{L^2(\Omega_f)}^2\,d\tau
-
\frac12\lVert X_0^N\rVert_{\mathbb H}^2
\Big]\psi(t)\,dt
=0.
\label{2.39FER}
\end{align}
Using \eqref{coul}, \eqref{OTep2}, and \eqref{conwl2}, we have
$$
X^N\rightharpoonup X
\qquad
\text{weakly in }L^2(0,T;\mathbb H).
$$
Since $\psi\ge0$, weak lower semicontinuity gives
\begin{align}
\int_0^T
\psi(t)\lVert X(t)\rVert_{\mathbb H}^2\,dt
\le
\liminf_{N\to\infty}
\int_0^T
\psi(t)\lVert X^N(t)\rVert_{\mathbb H}^2\,dt.
\label{2.4YwXC}
\end{align}
For the dissipation term in \eqref{2.39FER}, weak lower semicontinuity in \eqref{2.4YwXC} applied to
$$
\nabla u^N\rightharpoonup \nabla u
\qquad
\text{in }L^2(0,T;L^2(\Omega_f))
$$
which implies, for a.e. $t\in(0,T)$,
\begin{align}
\int_0^t
\lVert \nabla u(\tau)\rVert_{L^2(\Omega_f)}^2\,d\tau
\le
\liminf_{N\to\infty}
\int_0^t
\lVert \nabla u^N(\tau)\rVert_{L^2(\Omega_f)}^2\,d\tau.
\end{align}
Since $\psi\ge0$, Fatou's lemma yields
\begin{align}
\int_0^T
\psi(t)
\int_0^t
\lVert \nabla u(\tau)\rVert_{L^2(\Omega_f)}^2\,d\tau\,dt
&\le
\int_0^T
\psi(t)
\liminf_{N\to\infty}
\int_0^t
\lVert \nabla u^N(\tau)\rVert_{L^2(\Omega_f)}^2\,d\tau\,dt
\nonumber\\
&\le
\liminf_{N\to\infty}
\int_0^T
\psi(t)
\int_0^t
\lVert \nabla u^N(\tau)\rVert_{L^2(\Omega_f)}^2\,d\tau\,dt.
\label{2.42DS}
\end{align}
By \eqref{OPE56},
$
X_0^N\to X_0
\;
\text{strongly in }\mathbb H,$
therefore
\begin{align}
\lVert X_0^N\rVert_{\mathbb H}^2
\to
\lVert X_0\rVert_{\mathbb H}^2.
\label{ENL2.43}
\end{align}
Passing to the lower limit in \eqref{2.39FER} and using
\eqref{2.4YwXC}, \eqref{2.42DS}, and \eqref{ENL2.43}, we obtain
\begin{align}
\int_0^T
\Big[
\frac12\lVert X(t)\rVert_{\mathbb H}^2
+
\int_0^t
\lVert \nabla u(\tau)\rVert_{L^2(\Omega_f)}^2\,d\tau
-
\frac12\lVert X_0\rVert_{\mathbb H}^2
\Big]\psi(t)\,dt
\le0.
\label{2.44PSI}
\end{align}
Since $\psi\in C_c^\infty(0,T)$, $\psi\ge0$, is arbitrary, and using the fundamental lemma of the calculus of variations, \eqref{2.44PSI} implies
\begin{align}
\frac12\lVert X(t)\rVert_{\mathbb H}^2
+
\int_0^t
\lVert \nabla u(\tau)\rVert_{L^2(\Omega_f)}^2\,d\tau
\le
\frac12\lVert X_0\rVert_{\mathbb H}^2
\label{4GHKL}
\end{align}
for a.e. $t\in(0,T)$.
\medskip

\noindent \underline{\textbf{Weak continuity of the energy space:}} Define
$
Y(t):=(u(t),w_t(t)).
$
By \eqref{Hspace},
$
\lVert Y(t)\rVert_{\mathcal H}^2
=
\|u(t)\|_{L^2(\Omega_f)}^2
+
\|w_t(t)\|_{L^2(\Omega_s)}^2 .
$
Hence, by \eqref{4GHKL},
$
\|Y(t)\|_{\mathcal H}^2
\le
\|X(t)\|_{\mathbb H}^2 ,
$
and therefore
$
Y\in L^\infty(0,T;\mathcal H).
$
Moreover, \eqref{Zt-G} gives
\begin{align}
Y_t=\partial_t(u,w_t)\in L^2(0,T;\mathcal V').
\label{YtVprime}
\end{align}
Since
$
\mathcal H\hookrightarrow \mathcal V'
$
continuously, we also have
$
Y\in L^2(0,T;\mathcal V').
$
Thus
$
Y\in W^{1,2}(0,T;\mathcal V'),
$
and consequently
$
Y\in C([0,T];\mathcal V').
$

\noindent Since strong continuity implies weak continuity, we conclude that
$
Y\in C_s(0,T;\mathcal V').
$ We now apply Lions' Lemma
\cite[Lemma~8.1, Chapter~8]{LionsMagenes1972}
with
$
X=\mathcal H,
\;
Y=\mathcal V'.
$
Since $\mathcal H$ is a Hilbert space, it is reflexive. In addition,
$
\mathcal H\hookrightarrow \mathcal V'
$
continuously. Therefore,
$
Y\in L^\infty(0,T;\mathcal H)
\cap
C_s(0,T;\mathcal V')
$
implies
$
Y\in C_s(0,T;\mathcal H).
$
Equivalently,
$
Y\in C_w([0,T];\mathcal H).
$
That is,
\begin{align}
(u,w_t)
\in
C_w\big([0,T];
L^2(\Omega_f)\times L^2(\Omega_s)
\big).
\label{uwc}
\end{align}
We next establish the strong $L^2(\Omega_s)$-continuity of $w$. Since the Galerkin approximations satisfy
$
w_t^N=\partial_t w^N,
$
for every $\phi\in L^2(\Omega_s)$ and every
$\eta\in C_c^\infty(0,T)$ we have,
\begin{align}
-\int_0^T
(w^N(t),\phi)_{\Omega_s}\eta'(t)\,dt
=
\int_0^T
(w_t^N(t),\phi)_{\Omega_s}\eta(t)\,dt.
\label{wN-dist}
\end{align}
Passing to the limit in \eqref{wN-dist} by
\eqref{OTep2} and \eqref{conwl2}, w
we obtain
\begin{align}
-\int_0^T
(w(t),\phi)_{\Omega_s}\eta'(t)\,dt
=
\int_0^T
(w_t(t),\phi)_{\Omega_s}\eta(t)\,dt,
\label{w-dist}
\end{align}
from which we obtain
$$
\partial_t w=w_t
\qquad
\text{in }\mathcal D'(0,T;L^2(\Omega_s)).
$$
Since \eqref{w-dist} shows that $w_t$ is the distributional time derivative of $w$, define
\begin{align}
\widetilde w(t)
:=
w_0
+
\int_0^t w_t(s)\,ds,
\qquad
0\le t\le T.
\label{qe2.75}
\end{align}
Since
$
w_t\in L^\infty(0,T;L^2(\Omega_s))
$
and $T<\infty$, we have
$$
w_t\in L^1(0,T;L^2(\Omega_s)),
$$
so the integral in \eqref{qe2.75} is well defined as an
$L^2(\Omega_s)$-valued integral. Moreover, for any $r,t\in[0,T]$,
\begin{align}
\widetilde w(t)-\widetilde w(r)
&=
\int_r^t w_t(s)\,ds
\qquad
\text{in }L^2(\Omega_s),
\end{align}
which yields
\begin{align}
\lVert
\widetilde w(t)-\widetilde w(r)
\rVert_{L^2(\Omega_s)}
&\le
\int_r^t
\lVert w_t(s)\rVert_{L^2(\Omega_s)}\,ds.
\label{34KWxz}
\end{align}
Since
$
\lVert w_t(\cdot)\rVert_{L^2(\Omega_s)}
\in L^1(0,T),
$
it follows that
$$
\int_r^t
\lVert w_t(s)\rVert_{L^2(\Omega_s)}\,ds
\to0
\qquad
\text{as }\lvert t-r\rvert\to0,
$$
and therefore \eqref{34KWxz} implies
$$
\lVert
\widetilde w(t)-\widetilde w(r)
\rVert_{L^2(\Omega_s)}
\to0
\qquad
\text{as }\lvert t-r\rvert\to0,
$$
showing that
$
\widetilde w
$
is continuous as an
$
L^2(\Omega_s)
$
-valued function on
$
[0,T],
$
that is,
\begin{align}
\widetilde w\in C([0,T];L^2(\Omega_s)).
\end{align}
\noindent
We also record the distributional time derivative of $ \widetilde w$.
Let $\phi\in L^2(\Omega_s)$ be fixed and define
$
a(t):=(\widetilde w(t),\phi)_{\Omega_s},
\;
b(t):=(w_t(t),\phi)_{\Omega_s}.
$
It follows from \eqref{qe2.75} that
\begin{align}
a(t)
=
(w_0,\phi)_{\Omega_s}
+
\int_0^t b(s)\,ds .
\label{12C}
\end{align}
Since
$
w_t\in L^1(0,T;L^2(\Omega_s)),
$
we have
\begin{align}
\int_0^T \lvert b(s)\rvert\,ds
&=
\int_0^T
\lvert 
(w_t(s),\phi)_{\Omega_s}
\rvert\,ds
\nonumber\\
&\le
\int_0^T
\lVert w_t(s)\rVert_{L^2(\Omega_s)}
\lVert \phi\rVert_{L^2(\Omega_s)}
\,ds
\nonumber\\
&=
\lVert \phi\rVert_{L^2(\Omega_s)}
\int_0^T
\lVert w_t(s)\rVert_{L^2(\Omega_s)}
\,ds
<\infty,
\label{b-L1}
\end{align}
which implies that
$
b\in L^1(0,T).
$ Therefore, by the fundamental theorem of calculus for Lebesgue integrals,
\eqref{12C} implies that $a$ is absolutely continuous on $[0,T]$ and
$$
a'(t)=b(t)
\quad
\text{for a.e. }t\in(0,T).
$$
By the definition of weak derivative, for every
$\eta\in C_c^\infty(0,T)$,
\begin{align}
\int_0^T a(t)\eta'(t)\,dt
=
-\int_0^T a'(t)\eta(t)\,dt .
\end{align}
Substituting
$a'(t)=b(t)$,
we obtain
\begin{align}
-\int_0^T
(\widetilde w(t),\phi)_{\Omega_s}
\eta'(t)\,dt
&=
-\int_0^T
a(t)\eta'(t)\,dt
\nonumber\\
&=
\int_0^T
a'(t)\eta(t)\,dt
\nonumber\\
&=
\int_0^T
(w_t(t),\phi)_{\Omega_s}
\eta(t)\,dt .
\label{4.hy5}
\end{align}
Since $\phi\in L^2(\Omega_s)$ was arbitrary, it follows from \eqref{4.hy5} that
\begin{align}
\partial_t\widetilde w
=
w_t
\qquad
\text{in }
\mathcal D'(0,T;L^2(\Omega_s)).
\end{align}
\noindent
We now prove that the weak limit $w$ agrees with $\widetilde w$ as an
element of $L^2(0,T;L^2(\Omega_s))$. For every $\phi\in L^2(\Omega_s)$ and
every $\rho\in C_c^\infty(0,T)$, the identity
$
w_t^N=\partial_t w^N
$
together with
$
w^N(0)=w_0^N
$
gives
\begin{align}
w^N(t)=w_0^N+\int_0^t w_t^N(s)\,ds.
\label{v4rt}
\end{align}
Multiplying \eqref{v4rt} by $\rho(t)$, pairing with $\phi$, integrating
over $(0,T)$, and applying Fubini's theorem, we obtain
\begin{align}
\int_0^T
(w^N(t),\phi)_{\Omega_s}\rho(t)\,dt
&=
\int_0^T
(w_0^N,\phi)_{\Omega_s}\rho(t)\,dt
\nonumber\\
&\quad
+
\int_0^T
(w_t^N(s),\phi)_{\Omega_s}
\left(\int_s^T \rho(t)\,dt\right)\,ds.
\label{.2.84eq}
\end{align}
Passing to the limit in \eqref{.2.84eq}, using
\eqref{OTep2}, \eqref{conwl2} and $w_0^N \to w_0 \; \textrm{in} \; L^2(\Omega_s)$
we obtain
\begin{align}
\int_0^T
(w(t),\phi)_{\Omega_s}\rho(t)\,dt
&=
\int_0^T
(w_0,\phi)_{\Omega_s}\rho(t)\,dt
\nonumber\\
&\quad
+
\int_0^T
(w_t(s),\phi)_{\Omega_s}
\left(\int_s^T \rho(t)\,dt\right)\,ds.
\label{2.85WE}
\end{align}
By \eqref{qe2.75}, the right-hand side of \eqref{2.85WE} is equal to
$
\int_0^T
(\widetilde w(t),\phi)_{\Omega_s}\rho(t)\,dt,
$
therefore we have,
\begin{align}
\int_0^T
(w(t)-\widetilde w(t),\phi)_{\Omega_s}
\rho(t)\,dt
=0,
\label{FGT2}
\end{align}
for every
$
\phi\in L^2(\Omega_s)
$
and every
$
\rho\in C_c^\infty(0,T).
$
From \eqref{FGT2}, by the fundamental lemma of the calculus of variations, we obtain
$$
(w(t)-\widetilde w(t),\phi)_{\Omega_s}=0,
$$
for a.e. $t\in(0,T)$. Since
$
\phi\in L^2(\Omega_s)
$
was arbitrary, we conclude that
$$
w=\widetilde w
\qquad
\text{in }
L^2(0,T;L^2(\Omega_s)).
$$
From this point on, we identify $\widetilde w$ with $w$ and denote it again by $w$. Consequently,
\begin{align}
w(t)
=
w_0
+
\int_0^t w_t(s)\,ds
\qquad
\text{in }L^2(\Omega_s),
\quad
0\le t\le T,
\label{2.87EQW}
\end{align}
from which, for any $r,t\in[0,T]$, we obtain
\begin{align}
w(r)-w(t)
=
\int_t^r w_t(s)\,ds
\qquad
\text{in }L^2(\Omega_s).
\end{align}
Using
$
w_t\in L^\infty(0,T;L^2(\Omega_s))$, we obtain
\begin{align}
\lVert w(r)-w(t)\rVert_{L^2(\Omega_s)}
&\le
\int_t^r
\lVert w_t(s)\rVert_{L^2(\Omega_s)}\,ds
\nonumber\\
&\le
\lvert r-t\rvert
\lVert w_t\rVert_{L^\infty(0,T;L^2(\Omega_s))},
\label{cont-est}
\end{align}
which implies
\begin{align}
w\in C([0,T];L^2(\Omega_s)).
\label{wslc}
\end{align}
We now prove that
$
w\in C_w([0,T];H^1(\Omega_s)).
$
Since
$
w\in L^\infty(0,T;H^1(\Omega_s)),
$
we have
\begin{align}
\lVert w(r)\rVert_{H^1(\Omega_s)}
\le C
\qquad
\text{for a.e. }r\in(0,T). \label{vx2H1}\end{align}
Therefore, using \eqref{vx2H1}, for any fixed,
$
t\in[0,T],
$
we may choose a sequence
$
\{r_n\}\subset(0,T)
$
such that
$
r_n\to t
$
and
$$
\lVert w(r_n)\rVert_{H^1(\Omega_s)}
\le C
\qquad
\text{for every }n.
$$
In addition, since
$
\{w(r_n)\}_{n\ge1}
$
is bounded in the reflexive space
$
H^1(\Omega_s),$
 by Banach-Alaoglu, 
there exist a subsequence, not relabeled, and a function
$
z\in H^1(\Omega_s)
$
such that
\begin{align}
w(r_n)
\rightharpoonup
z
\qquad
\text{weakly in }
H^1(\Omega_s).
\label{2.7g1}
\end{align}
Since the embedding $H^1(\Omega_s)\hookrightarrow L^2(\Omega_s)$ is compact,
the same subsequence satisfies
\begin{align}
w(r_n)\to z
\qquad
\text{strongly in }L^2(\Omega_s).
\label{sa2}
\end{align}
On the other hand, since
$
w\in C([0,T];L^2(\Omega_s))$ from \eqref{wslc} and $r_n\to t$, we have
\begin{align}
w(r_n)\to w(t)
\qquad
\text{strongly in }L^2(\Omega_s).
\label{cdq}
\end{align}
Therefore, it follows from \eqref{sa2} and \eqref{cdq} that
$$
z=w(t)
\qquad
\text{in }L^2(\Omega_s).
$$
Since
$
z=w(t)
\;
\text{in }L^2(\Omega_s),
$
and
$
z\in H^1(\Omega_s),
$
we identify $w(t)$ with $z$ as an element of
$
H^1(\Omega_s).
$  Moreover, using \eqref{2.7g1}, by weak lower semicontinuity,
$$
\lVert w(t)\rVert_{H^1(\Omega_s)}=\lVert z\rVert_{H^1(\Omega_s)}
\le
\liminf_{n\to\infty}
\lVert w(r_n)\rVert_{H^1(\Omega_s)}
\le C.
$$
Since $t\in[0,T]$ was arbitrary, we obtain
\begin{align}
\sup_{t\in[0,T]}
\lVert w(t)\rVert_{H^1(\Omega_s)}
\le C.
\label{12S7}
\end{align}
To prove weak continuity, fix $t\in[0,T]$ and let $\{t_n\}\subset[0,T]$ be an arbitrary sequence such that $t_n\to t$. By \eqref{12S7},
$
\{w(t_n)\}_{n\ge1}
$
is bounded in $H^1(\Omega_s)$. Let
$
\{w(t_{n_k})\}_{k\ge1}
$
be an arbitrary subsequence of
$
\{w(t_n)\}_{n\ge1}.
$
Since
$
\{w(t_{n_k})\}_{k\ge1}
$
is bounded in the reflexive space
$
H^1(\Omega_s),
$
there exist a further subsequence
$
\{w(t_{n_{k_j}})\}_{j\ge1}
$
and a function
$
z\in H^1(\Omega_s)
$
such that
\begin{align}
w(t_{n_{k_j}})
\rightharpoonup z
\qquad
\text{weakly in }H^1(\Omega_s).\label{sam}
\end{align}
By the compact embedding,
$
H^1(\Omega_s)\hookrightarrow L^2(\Omega_s),
$ it follows from \eqref{sam} that
\begin{align}
w(t_{n_{k_j}})
\to z
\qquad
\text{strongly in }L^2(\Omega_s).
\label{z12a}
\end{align}
Since
$
w\in C([0,T];L^2(\Omega_s))
$
from \eqref{wslc}, it follows that \begin{align}
w(t_{n_{k_j}})
\to
w(t)
\qquad
\text{strongly in }L^2(\Omega_s).
\label{cdxz}
\end{align}
Therefore, from \eqref{z12a} and \eqref{cdxz} we obtain
\begin{align}
z=w(t)
\qquad
\text{in }L^2(\Omega_s).
\label{lx2}
\end{align}
Thus, by \eqref{sam} and \eqref{lx2},
$$
w(t_{n_{k_j}})
\rightharpoonup
w(t)
\qquad
\text{weakly in }H^1(\Omega_s),
$$
which shows that every subsequence
$
\{w(t_{n_k})\}_{k\ge1}
$
of
$
\{w(t_n)\}_{n\ge1}
$
contains a further subsequence
$
\{w(t_{n_{k_j}})\}_{j\ge1}
$
converging weakly in
$
H^1(\Omega_s)
$
to
$
w(t).
$
Hence every subsequence
$
\{w(t_{n_k})\}_{k\ge1}
$
of
$
\{w(t_n)\}_{n\ge1}
$
contains a further subsequence
$
\{w(t_{n_{k_j}})\}_{j\ge1}
$
converging weakly in
$
H^1(\Omega_s)
$
to
$
w(t).
$
Consequently,
\begin{align}
w(t_n)
\rightharpoonup
w(t)
\qquad
\text{weakly in }H^1(\Omega_s).
\label{78bf}
\end{align}
Since
$
t_n\to t
$
was arbitrary, we conclude from \eqref{78bf} that
\begin{align}
w\in C_w([0,T];H^1(\Omega_s)).
\label{wwh1}
\end{align}

\noindent Define
$
X(t):=(u(t),w(t),w_t(t)).
$
From \eqref{uwc} and \eqref{wwh1}, we obtain
\begin{align}
X\in C_w([0,T];\mathbb H).
\label{xwc}
\end{align}
\underline{\textbf{Initial condition:}}  We first identify $u(0)$ and $w_t(0)$. From
\eqref{lire}, for every $(\varphi_f,\varphi_s)\in\mathcal V$ and every
$\psi\in C^\infty([0,T])$ with $ \psi(T)=0$, we have
\begin{align}
0
&=
-\int_0^T
(u,\varphi_f)_{\Omega_f}\psi'(t)\,dt
-
(u_0,\varphi_f)_{\Omega_f}\psi(0)
\nonumber\\
&\quad
-
\int_0^T
(w_t,\varphi_s)_{\Omega_s}\psi'(t)\,dt
-
(w_1,\varphi_s)_{\Omega_s}\psi(0)
\nonumber\\
&\quad
+
\int_0^T
(\nabla u,\nabla\varphi_f)_{\Omega_f}\psi(t)\,dt
+
\int_0^T
(\nabla w,\nabla\varphi_s)_{\Omega_s}\psi(t)\,dt
+
\int_0^T
(w,\varphi_s)_{\Omega_s}\psi(t)\,dt .
\label{inL}
\end{align}
On the other hand, by \eqref{WNCW}, multiplying by $\psi$ and integrating
over $(0,T)$, we obtain
\begin{align}
0
&=
\int_0^T
\langle
\partial_t(u,w_t)(t),
(\varphi_f,\varphi_s)
\rangle_{\mathcal V',\mathcal V}
\psi(t)\,dt
\nonumber\\
&\quad
+
\int_0^T
(\nabla u,\nabla\varphi_f)_{\Omega_f}\psi(t)\,dt
+
\int_0^T
(\nabla w,\nabla\varphi_s)_{\Omega_s}\psi(t)\,dt
+
\int_0^T
(w,\varphi_s)_{\Omega_s}\psi(t)\,dt .
\label{wft}
\end{align}
Since \eqref{uwc} and \eqref{YtVprime} imply
$
Y=(u,w_t)\in C_w([0,T];\mathcal H)
$
and
$
Y_t=\partial_t(u,w_t)\in L^2(0,T;\mathcal V'),
$
the initial values
$u(0)$
and
$w_t(0)
$
are well defined in the weak topology of
$
\mathcal H,
$
the weak integration-by-parts formula gives
\begin{align}
\int_0^T
\left\langle
\partial_t(u,w_t)(t),
(\varphi_f,\varphi_s)
\right\rangle_{\mathcal V',\mathcal V}
\psi(t)\,dt
&=
-\int_0^T
\Big[
(u,\varphi_f)_{\Omega_f}
+
(w_t,\varphi_s)_{\Omega_s}
\Big]\psi'(t)\,dt
\nonumber\\
&\quad
-
\Big[
(u(0),\varphi_f)_{\Omega_f}
+
(w_t(0),\varphi_s)_{\Omega_s}
\Big]\psi(0).
\label{IBP-Y}
\end{align}
Substituting \eqref{IBP-Y} into \eqref{wft}, we get
\begin{align}
0
&=
-\int_0^T
\Big[
(u,\varphi_f)_{\Omega_f}
+
(w_t,\varphi_s)_{\Omega_s}
\Big]\psi'(t)\,dt
\nonumber\\
&\quad
-
\Big[
(u(0),\varphi_f)_{\Omega_f}
+
(w_t(0),\varphi_s)_{\Omega_s}
\Big]\psi(0)
\nonumber\\
&\quad
+
\int_0^T
(\nabla u,\nabla\varphi_f)_{\Omega_f}\psi(t)\,dt
+
\int_0^T
(\nabla w,\nabla\varphi_s)_{\Omega_s}\psi(t)\,dt
+
\int_0^T
(w,\varphi_s)_{\Omega_s}\psi(t)\,dt .
\label{wft-IBP}
\end{align}
Comparing \eqref{wft-IBP} with \eqref{inL}, the integral terms cancel and we
obtain
\begin{align}
\Big[
(u(0),\varphi_f)_{\Omega_f}
+
(w_t(0),\varphi_s)_{\Omega_s}
\Big]\psi(0)
=
\Big[
(u_0,\varphi_f)_{\Omega_f}
+
(w_1,\varphi_s)_{\Omega_s}
\Big]\psi(0).
\label{F4a1}
\end{align}
Taking $\psi\in C^\infty([0,T])$ with $\psi(T)=0$ and $\psi(0)\neq0$,
we conclude from \eqref{F4a1} that
\begin{align}
\big(u(0)-u_0,\varphi_f\big)_{\Omega_f}
+
\big(w_t(0)-w_1,\varphi_s\big)_{\Omega_s}
=
0
\label{iid}
\end{align}
for every
$
(\varphi_f,\varphi_s)\in\mathcal V.
$

\noindent To identify the fluid initial value, take
$
\varphi_f\in C_c^\infty(\Omega_f)
$
and
$
\varphi_s=0.
$
Then the trace of $\varphi_f$ on $\Gamma_s$ is zero, and hence
$
(\varphi_f,0)\in\mathcal V.
$
Therefore \eqref{iid} gives
\begin{align}
\big(u(0)-u_0,\varphi_f\big)_{\Omega_f}=0
\qquad
\forall \varphi_f\in C_c^\infty(\Omega_f).
\label{lk3s}
\end{align}
Since $C_c^\infty(\Omega_f)$ is dense in $L^2(\Omega_f)$, it follows from \eqref{lk3s} that
\begin{align}
u(0)=u_0
\qquad
\text{in }L^2(\Omega_f).
\label{ui}
\end{align}
Similarly, take
$
\varphi_f=0
$
and
$
\varphi_s\in C_c^\infty(\Omega_s).
$
Then the trace of $\varphi_s$ on $\Gamma_s$ is zero, and hence
$
(0,\varphi_s)\in\mathcal V.
$
Thus \eqref{iid} gives
\begin{align}
\big(w_t(0)-w_1,\varphi_s\big)_{\Omega_s}=0
\qquad
\forall \varphi_s\in C_c^\infty(\Omega_s).
\label{b12z}
\end{align}
Since $C_c^\infty(\Omega_s)$ is dense in $L^2(\Omega_s)$, it follows from \eqref{b12z} that
\begin{align}
w_t(0)=w_1
\qquad
\text{in }L^2(\Omega_s).
\label{wti}
\end{align}

\noindent It remains to identify $w(0)$. From \eqref{2.87EQW}, taking
$t=0$, we have
\begin{align}
w(0)
=
w_0
+
\int_0^0 w_t(s)\,ds
=
w_0
\qquad
\text{in }L^2(\Omega_s).
\label{wl2}
\end{align}
We now upgrade this equality to $H^1(\Omega_s)$. By \eqref{wwh1},
$
w(0)\in H^1(\Omega_s),
$
and by the assumption
$
X_0\in\mathbb H
$
we have
$
w_0\in H^1(\Omega_s).
$
Define
$
\zeta:=w(0)-w_0
$, 
then
$
\zeta\in H^1(\Omega_s).
$
On the other hand, \eqref{wl2} gives
\begin{align}
\zeta=0
\qquad
\text{in }L^2(\Omega_s).
\label{2.107EO}
\end{align}
Since
$
\zeta\in H^1(\Omega_s),
$
for every
$
\phi\in C_c^\infty(\Omega_s)
$
and every
$
i=1,2,
$
the definition of weak derivative gives
\begin{align}
\int_{\Omega_s}
\zeta\,\partial_i\phi\,dx
=
-
\int_{\Omega_s}
\partial_i\zeta\,\phi\,dx.
\label{TF3}
\end{align}
Since, by \eqref{2.107EO} $\zeta=0\,\text{in }L^2(\Omega_s)
$, the left-hand side of \eqref{TF3} vanishes.
Therefore, from \eqref{TF3}, we obtain
\begin{align}
\int_{\Omega_s}
\partial_i\zeta\,\phi\,dx
=
0
\qquad
\forall \phi\in C_c^\infty(\Omega_s),
\end{align}
which implies
\begin{align}
\partial_i\zeta=0
\qquad
\text{in }L^2(\Omega_s),
\qquad
i=1, 2.
\end{align}
and therefore
\begin{align}
\nabla\zeta=0
\qquad
\text{in }L^2(\Omega_s).
\label{2.107EL}
\end{align}
Hence \eqref{2.107EO} and \eqref{2.107EL} together implies that \begin{align}\label{2.103EQL}
w(0)=w_0 \qquad \textrm{in}\; H^1(\Omega_s).
\end{align}
Combining \eqref{ui}, \eqref{wti}, and \eqref{2.103EQL}, we conclude that
\begin{align}
X(0)=X_0
\qquad
\text{in }\mathbb H.
\label{2.1033E}
\end{align}
\noindent
\underline{\textbf{Strong continuity at the initial time:}} From \eqref{xwc} and \eqref{2.1033E}, we have
\begin{align}
X(t)\rightharpoonup X_0
\qquad
\text{weakly in }\mathbb H
\qquad
\text{as }t\to0^+.
\label{weak-at-zero}
\end{align}
It follows from \eqref{4GHKL}, for a.e. $t\in(0,T)$ that 
\begin{align}
\frac12\lVert X(t)\rVert_{\mathbb H}^2
+
\int_0^t
\lVert \nabla u(\tau)\rVert_{L^2(\Omega_f)}^2\,d\tau
\le
\frac12\lVert X_0\rVert_{\mathbb H}^2.
\label{energy-ae}
\end{align}
We first extend the estimate to every $t\in[0,T]$. Fix $t\in[0,T]$. Since \eqref{energy-ae} holds for a.e. $t\in(0,T)$, we may choose a sequence
$
\{t_n\}\subset(0,T)
$
such that
$
t_n\to t
$
and \eqref{energy-ae} holds at each $t_n$. Hence, from \eqref{energy-ae}, we have
\begin{align}
\frac12\lVert X(t_n)\rVert_{\mathbb H}^2
+
\int_0^{t_n}
\lVert \nabla u(\tau)\rVert_{L^2(\Omega_f)}^2\,d\tau
\le
\frac12\lVert X_0\rVert_{\mathbb H}^2.
\label{er4}
\end{align}
Since $X\in C_w([0,T];\mathbb H)$ by \eqref{xwc}, we have
$
X(t_n)\rightharpoonup X(t)
\;
\text{weakly in }\mathbb H.
$
Therefore, by weak lower semicontinuity,
\begin{align}
\frac12\lVert X(t)\rVert_{\mathbb H}^2
\le
\liminf_{n\to\infty}
\frac12\lVert X(t_n)\rVert_{\mathbb H}^2.
\label{2.110SD}
\end{align}
In addition, since
$
\nabla u\in L^2(0,T;L^2(\Omega_f)),
$
we have
$
\lVert \nabla u(\tau)\rVert_{L^2(\Omega_f)}^2\in L^1(0,T).
$
Therefore, the map
$
r\mapsto
\int_0^r
\lVert \nabla u(\tau)\rVert_{L^2(\Omega_f)}^2\,d\tau
$
is continuous on $[0,T]$, and hence
\begin{align}
\int_0^{t_n}
\lVert \nabla u(\tau)\rVert_{L^2(\Omega_f)}^2\,d\tau
\to
\int_0^t
\lVert \nabla u(\tau)\rVert_{L^2(\Omega_f)}^2\,d\tau.
\label{2.111SD}
\end{align}
 Applying $\liminf_{n\to\infty}$ to \eqref{er4} and 
using \eqref{2.110SD} and \eqref{2.111SD}, we obtain
\begin{align}
\frac12\lVert X(t)\rVert_{\mathbb H}^2
+
\int_0^t
\lVert \nabla u(\tau)\rVert_{L^2(\Omega_f)}^2\,d\tau
\le
\frac12\lVert X_0\rVert_{\mathbb H}^2.
\label{2.106EL}
\end{align}
Since $t\in[0,T]$ was arbitrary, \eqref{2.106EL} holds for every
$t\in[0,T]$. In particular, \eqref{2.106EL} implies
\begin{align}
\lVert X(t)\rVert_{\mathbb H}^2
\le
\lVert X_0\rVert_{\mathbb H}^2
\qquad
\text{for every }t\in[0,T],
\end{align}
 and hence
\begin{align}
\limsup_{t\to0^+}
\lVert X(t)\rVert_{\mathbb H}^2
\le
\lVert X_0\rVert_{\mathbb H}^2.
\label{lie}
\end{align}
On the other hand, from \eqref{weak-at-zero} and weak lower semicontinuity,
\begin{align}
\lVert X_0\rVert_{\mathbb H}^2
\le
\liminf_{t\to0^+}
\lVert X(t)\rVert_{\mathbb H}^2.
\label{limfe}
\end{align}
Combining \eqref{lie} and \eqref{limfe}, we get
\begin{align}
\lim_{t\to0^+}
\lVert X(t)\rVert_{\mathbb H}
=
\lVert X_0\rVert_{\mathbb H}.
\label{nc}
\end{align}
Since
$
X(t)\rightharpoonup X_0
\;
\text{weakly in }\mathbb H
\;
\text{as }t\to0^+
$
by \eqref{weak-at-zero}, and since the norms converge by \eqref{nc}, we obtain
$$
\lVert X(t)-X_0\rVert_{\mathbb H}^2
=
\lVert X(t)\rVert_{\mathbb H}^2
+
\lVert X_0\rVert_{\mathbb H}^2
-
2(X(t),X_0)_{\mathbb H},
$$
which implies
\begin{align}
X(t)\to X_0
\qquad
\text{strongly in }\mathbb H
\qquad
\text{as }t\to0^+.
\label{siH}
\end{align}
\noindent
\underline{\textbf{Uniqueness of the Galerkin-limit solution:}} 
Let
$
X_N^{(i)}(t)
=
\bigl(u_N^{(i)}(t),w_N^{(i)}(t),w_{N,t}^{(i)}(t)\bigr),
\; i=1,2,
$
be two Galerkin approximations in the same space $\mathcal V_N$, with the
same initial data. 
Define
$$
\widehat u_N:=u_N^{(1)}-u_N^{(2)},
\qquad
\widehat w_N:=w_N^{(1)}-w_N^{(2)},
\qquad
\widehat z_N:=w_{N,t}^{(1)}-w_{N,t}^{(2)}.
$$
Define $\widehat{X}_N:=(\widehat{u}_N, \widehat{w}_N, \widehat{z}_N)$ and
$
\widehat X:=X^{(1)}-X^{(2)}
=
(\widehat u,\widehat w,\widehat z);
$
then, by construction,
$
\widehat z_N=\partial_t\widehat w_N
$
and
$
\widehat X_N(0)=0.
$
Subtracting the two Galerkin systems gives, for every
$(\varphi_f,\varphi_s)\in\mathcal V_N$,
\begin{align}
(\partial_t\widehat u_N,\varphi_f)_{\Omega_f}
+
(\nabla\widehat u_N,\nabla\varphi_f)_{\Omega_f}
+
(\partial_t\widehat z_N,\varphi_s)_{\Omega_s}
+
(\nabla\widehat w_N,\nabla\varphi_s)_{\Omega_s}
+
(\widehat w_N,\varphi_s)_{\Omega_s}
=
0.
\label{2.112AQ}
\end{align}
Since
$
(\widehat u_N(t),\widehat z_N(t))\in\mathcal V_N$, we may take
$
(\varphi_f,\varphi_s)=(\widehat u_N,\widehat z_N)
$
in \eqref{2.112AQ}. Hence
\begin{align}
\frac12\frac{d}{dt}
\left(
\lVert \widehat u_N\rVert_{L^2(\Omega_f)}^2
+
\lVert \widehat z_N\rVert_{L^2(\Omega_s)}^2
+
\lvert \nabla\widehat w_N\rVert_{L^2(\Omega_s)}^2
+
\lVert \widehat w_N\rVert_{L^2(\Omega_s)}^2
\right)
+
\lVert \nabla\widehat u_N\rVert_{L^2(\Omega_f)}^2
=
0.
\label{2.114GE}
\end{align}
Integrating over $(0,t)$, using $\widehat X_N(0)=0$, gives
\begin{align}
\frac12\lVert \widehat X_N(t)\rVert_{\mathbb H}^2
+
\int_0^t
\lvert \nabla\widehat u_N(\tau)\rVert_{L^2(\Omega_f)}^2\,d\tau
=
0.
\label{re212}
\end{align}
Since both terms on the left-hand side are nonnegative, it follows that
\begin{align}
\widehat {X}_{N}(t)=0
\qquad
\text{in }\mathbb H
\qquad
\forall t\in[0,T].
\label{we23v}
\end{align}
Therefore,
$$
X_N^{(1)}(t)=X_N^{(2)}(t)
\qquad
\forall t\in[0,T].
$$
Thus, for each fixed $N$, the Galerkin approximation is unique.
\end{proof}

\subsection{A priori estimates and the resulting $H^3$-regularity}

\noindent
We first derive the tangential and time-differentiated a priori estimates in Proposition \ref{Prop 2.2} which form the foundation for the higher regularity analysis. The estimates obtained below provide uniform control of tangential spatial derivatives together with successive time derivatives of the coupled heat-wave system \eqref{1.0}. These bounds will later be combined with the equations themselves to recover the missing normal derivatives from the already controlled tangential and time-differentiated quantities, yielding higher regularity up to the $H^3$-level in Theorem \ref{Th 2.3}.
\begin{proposition}
\label{Prop 2.2}
Let $s\ge 3$. Assume that
\begin{align}
D^\alpha u_0&\in L^2(\Omega_f),
\qquad 0\le\alpha\le s,
\\
D^\alpha w_0&\in H^1(\Omega_s),
\qquad 0\le\alpha\le s,
\\
D^\alpha w_1&\in L^2(\Omega_s),
\qquad 0\le\alpha\le s,
\end{align}
 first-order time derivatives of the initial data satisfy
\begin{align}
D^\alpha u_t(0)&\in L^2(\Omega_f),
\qquad 0\le\alpha\le s-1,
\\
D^\alpha w_t(0)&\in H^1(\Omega_s),
\qquad 0\le\alpha\le s-1,
\\
D^\alpha w_{tt}(0)&\in L^2(\Omega_s),
\qquad 0\le\alpha\le s-1,
\end{align}
and the second-order time derivatives of the  initial data satisfy
\begin{align}
D^\alpha u_{tt}(0)&\in L^2(\Omega_f),
\qquad 0\le\alpha\le s-2,
\\
D^\alpha w_{tt}(0)&\in H^1(\Omega_s),
\qquad 0\le\alpha\le s-2,
\\
D^\alpha w_{ttt}(0)&\in L^2(\Omega_s),
\qquad 0\le\alpha\le s-2.
\end{align}
Then the corresponding solution $(w, w_t, u)$ to \eqref{1.0} satisfies the following estimates:
\begin{align}
\sup_{t\in[0,T]}
\sum_{\alpha=0}^{s}
\Big(
&\lVert D^\alpha u(t)\rVert_{L^2(\Omega_f)}^2
+
\lVert D^\alpha w_t(t)\rVert_{L^2(\Omega_s)}^2
+
\lVert D^\alpha w(t)\rVert_{H^1(\Omega_s)}^2
\Big)
\nonumber\\
&\qquad
+
\int_0^T
\sum_{\alpha=0}^{s}
\lVert \nabla D^\alpha u(t)\rVert_{L^2(\Omega_f)}^2\,dt
\le C,
\label{BHs}
\end{align}
\begin{align}
\sup_{t\in[0,T]}
\sum_{\alpha=0}^{s-1}
\Big(
&\lVert D^\alpha u_t(t)\rVert_{L^2(\Omega_f)}^2
+
\lVert D^\alpha w_{tt}(t)\rVert_{L^2(\Omega_s)}^2
+
\lVert D^\alpha w_t(t)\rVert_{H^1(\Omega_s)}^2
\Big)
\nonumber\\
&\qquad
+
\int_0^T
\sum_{\alpha=0}^{s-1}
\lVert \nabla D^\alpha u_t(t)\rVert_{L^2(\Omega_f)}^2\,dt
\le C,
\label{LemmaA}
\end{align}
and
\begin{align}
\sup_{t\in[0,T]}
\sum_{\alpha=0}^{s-2}
\Big(
&\lVert D^\alpha w_{ttt}(t)\rVert_{L^2(\Omega_s)}^2
+
\lVert D^\alpha w_{tt}(t)\rVert_{H^1(\Omega_s)}^2
+
\lVert  D^\alpha u_{tt}(t)\rVert_{L^2(\Omega_f)}^2
\Big)
\nonumber\\
&\qquad
+
\int_0^T
\sum_{\alpha=0}^{s-2}
\lVert \nabla D^\alpha u_{tt}(t)\rVert_{L^2(\Omega_f)}^2\,dt
\le C.
\label{E2.130s}
\end{align}
\end{proposition}
\begin{proof}[Proof of Proposition \ref{Prop 2.2}]
Here $D^\alpha$ denotes tangential derivatives along the periodic direction.
Since the coefficients are constant and $\Gamma_s$ is flat,
$$
[D^\alpha,\Delta]=0,
\qquad
[D^\alpha,\partial_\nu]=0, \qquad \textrm{ and} \qquad 
(D^\alpha v)\rvert_{\Gamma_s}
=
D^\alpha(v\rvert_{\Gamma_s}).
$$
Consequently, applying $D^\alpha$ to the equations and coupling conditions
preserves their form.

\noindent For $0\le\alpha\le s$, applying $D^\alpha$ to \eqref{1.0a}-\eqref{1.0b} gives
\begin{align}
D^\alpha u_t-\Delta D^\alpha u&=0
\qquad
\text{in }\Omega_f,
\label{XY}
\\
D^\alpha w_{tt}-\Delta D^\alpha w+D^\alpha w&=0
\qquad
\text{in }\Omega_s.
\label{YZ}
\end{align}
The corresponding trace and normal derivative relations become
\begin{align}
D^\alpha u&=0
\qquad
\text{on }\Gamma_f,
\\
D^\alpha w&=0
\qquad
\text{on }\Gamma_{s,0},
\\
D^\alpha u&=D^\alpha w_t
\qquad
\text{on }\Gamma_s,
\label{eqtf}
\\
\partial_\nu D^\alpha u&=\partial_\nu D^\alpha w
\qquad
\text{on }\Gamma_s.
\label{tncft}
\end{align}
Testing \eqref{XY} by $D^\alpha u$ and
\eqref{YZ} by $D^\alpha w_t$, integrating by parts, and using \eqref{eqtf}-\eqref{tncft}, we obtain
\begin{align}
\frac12\frac{d}{dt}
\Big(
&\lVert D^\alpha u\rVert_{L^2(\Omega_f)}^2
+
\lVert D^\alpha w_t\rVert_{L^2(\Omega_s)}^2
+
\lVert D^\alpha w\rVert_{H^1(\Omega_s)}^2
\Big)
+
\lVert \nabla D^\alpha u\rVert_{L^2(\Omega_f)}^2
=0.
\label{BA}
\end{align}
Integrating \eqref{BA} over $(0,t)$ yields
\begin{align}
&\lVert D^\alpha u(t)\rVert_{L^2(\Omega_f)}^2
+
\lVert D^\alpha w_t(t)\rVert_{L^2(\Omega_s)}^2
+
\lVert D^\alpha w(t)\rVert_{H^1(\Omega_s)}^2
\nonumber\\
&\qquad
+
2\int_0^t
\lVert \nabla D^\alpha u(\tau)\rVert_{L^2(\Omega_f)}^2\,d\tau
\nonumber\\
&=
\lVert D^\alpha u_0\rVert_{L^2(\Omega_f)}^2
+
\lVert D^\alpha w_1\rVert_{L^2(\Omega_s)}^2
+
\lVert D^\alpha w_0\rVert_{H^1(\Omega_s)}^2.
\label{BH-int}
\end{align}
Summing \eqref{BH-int} over $0\le\alpha\le s$ gives \eqref{BHs}.
Next, differentiating \eqref{XY}-\eqref{YZ} once in time gives
\begin{align}
D^\alpha u_{tt}-\Delta D^\alpha u_t&=0
\qquad
\text{in }\Omega_f,
\label{ALP}
\\
D^\alpha w_{ttt}-\Delta D^\alpha w_t+D^\alpha w_t&=0
\qquad
\text{in }\Omega_s.
\label{QWE}
\end{align}
The corresponding trace and normal derivative relations become
\begin{align}
D^\alpha u_t&=0
\qquad
\text{on }\Gamma_f,
\\
D^\alpha w_t&=0
\qquad
\text{on }\Gamma_{s,0},
\\
D^\alpha u_t&=D^\alpha w_{tt}
\qquad
\text{on }\Gamma_s,
\label{TSD}
\\
\partial_\nu D^\alpha u_t&=\partial_\nu D^\alpha w_t
\qquad
\text{on }\Gamma_s.
\label{TNCD}
\end{align}
Testing \eqref{ALP} by $D^\alpha u_t$ and
\eqref{QWE} by $D^\alpha w_{tt}$, integrating by parts, and using \eqref{TSD}-\eqref{TNCD}, we obtain
\begin{align}
\frac12\frac{d}{dt}
\Big(
&\lVert D^\alpha u_t\rVert_{L^2(\Omega_f)}^2
+
\lVert D^\alpha w_{tt}\rVert_{L^2(\Omega_s)}^2
+
\lVert D^\alpha w_t\rVert_{H^1(\Omega_s)}^2
\Big)
+
\lVert \nabla D^\alpha u_t\rVert_{L^2(\Omega_f)}^2
=0.
\label{TAA}
\end{align}
Integrating \eqref{TAA} over $(0,t)$ gives
\begin{align}
&\lVert D^\alpha u_t(t)\rVert_{L^2(\Omega_f)}^2
+
\lVert D^\alpha w_{tt}(t)\rVert_{L^2(\Omega_s)}^2
+
\lVert D^\alpha w_t(t)\rVert_{H^1(\Omega_s)}^2
\nonumber\\
&\qquad
+
2\int_0^t
\lVert \nabla D^\alpha u_t(\tau)\rVert_{L^2(\Omega_f)}^2\,d\tau
\nonumber\\
&=
\lVert D^\alpha u_t(0)\rVert_{L^2(\Omega_f)}^2
+
\lVert D^\alpha w_{tt}(0)\rVert_{L^2(\Omega_s)}^2
+
\lVert D^\alpha w_t(0)\rVert_{H^1(\Omega_s)}^2.
\label{TYR}
\end{align}
Summing \eqref{TYR} over $0\le\alpha\le s-1$ gives \eqref{LemmaA}.

\noindent Finally, for $0\le\alpha\le s-2$, differentiating
\eqref{ALP} once more in time and
\eqref{QWE} once more in time gives
\begin{align}
D^\alpha u_{ttt}-\Delta D^\alpha u_{tt}&=0
\qquad
\text{in }\Omega_f,
\label{tncfd}
\\
D^\alpha w_{tttt}-\Delta D^\alpha w_{tt}+D^\alpha w_{tt}&=0
\qquad
\text{in }\Omega_s.
\label{tnrs}
\end{align}
The corresponding trace and normal derivative relations become
\begin{align}
D^\alpha u_{tt}&=0
\qquad
\text{on }\Gamma_f,
\\
D^\alpha w_{tt}&=0
\qquad
\text{on }\Gamma_{s,0},
\\
D^\alpha u_{tt}&=D^\alpha w_{ttt}
\qquad
\text{on }\Gamma_s,
\label{AWET}
\\
\partial_\nu D^\alpha u_{tt}&=\partial_\nu D^\alpha w_{tt}
\qquad
\text{on }\Gamma_s.
\label{tnctnmt}
\end{align}
Testing \eqref{tncfd} by $D^\alpha u_{tt}$ and
\eqref{tnrs} by $D^\alpha w_{ttt}$, integrating by parts, and using \eqref{AWET}-\eqref{tnctnmt}, we obtain
\begin{align}
\frac12\frac{d}{dt}
\Big(
&\lVert D^\alpha w_{ttt}\rVert_{L^2(\Omega_s)}^2
+
\lVert D^\alpha w_{tt}\rVert_{H^1(\Omega_s)}^2
+
\lVert D^\alpha u_{tt}\rVert_{L^2(\Omega_f)}^2
\Big)
+
\lVert \nabla D^\alpha u_{tt}\rVert_{L^2(\Omega_f)}^2
=0.
\label{2.153HA}
\end{align}
Integrating \eqref{2.153HA} over $(0,t)$ gives
\begin{align}
&\lVert D^\alpha w_{ttt}(t)\rVert_{L^2(\Omega_s)}^2
+
\lVert D^\alpha w_{tt}(t)\rVert_{H^1(\Omega_s)}^2
+
\lVert D^\alpha u_{tt}(t)\rVert_{L^2(\Omega_f)}^2
\nonumber\\
&\qquad
+
2\int_0^t
\lVert \nabla D^\alpha u_{tt}(\tau)\rVert_{L^2(\Omega_f)}^2\,d\tau
\nonumber\\
&=
\lVert D^\alpha w_{ttt}(0)\rVert_{L^2(\Omega_s)}^2
+
\lVert D^\alpha w_{tt}(0)\rVert_{H^1(\Omega_s)}^2
+
\lVert D^\alpha u_{tt}(0)\rVert_{L^2(\Omega_f)}^2.
\label{2.154HI}
\end{align}
Summing \eqref{2.154HI} over $0\le\alpha\le s-2$ gives \eqref{E2.130s}. The proof is complete.
\end{proof}
\begin{theorem} \label{Th 2.3}
Assume Proposition \ref{Prop 2.2} holds. Then the estimates
\eqref{BHs}, \eqref{LemmaA}, and \eqref{E2.130s}
imply
\begin{align}
w&\in L^\infty(0,T;H^3(\Omega_s)),
\\
w_t&\in L^\infty(0,T;H^2(\Omega_s)),
\\
u&\in 
L^2(0,T;H^3(\Omega_f)).
\end{align}
\end{theorem}
\begin{proof}
\noindent
First, we prove
$
w\in L^\infty(0,T;H^3(\Omega_s)).
$
It is enough to control all derivatives
$
\partial_{x_1}^{\alpha}\partial_{x_2}^b w
$
with
$
\alpha+b \le 3
$
in
$
L^\infty(0,T;L^2(\Omega_s)).
$ From \eqref{BHs}, we have
\begin{align}
\sup_{t\in[0,T]}
\sum_{\alpha=0}^{s}
\lVert D^\alpha w(t)\rVert_{H^1(\Omega_s)}^2
\le C.
\label{basic}
\end{align}
Since
$
D^\alpha=\partial_{x_1}^\alpha,
$
\eqref{basic} gives
\begin{align}
\sup_{t\in[0,T]}
\lVert \partial_{x_1}^{\alpha} w(t)\rVert_{L^2(\Omega_s)}
&\le C,
\qquad 0\le \alpha \le s+1,
\label{TLO}
\\
\sup_{t\in[0,T]}
\lVert \partial_{x_1}^{\alpha}\partial_{x_2} w(t)\rVert_{L^2(\Omega_s)}
&\le C,
\qquad 0\le \alpha\le s.
\label{onPL}
\end{align}
Using \eqref{1.0b} and
$
\Delta=\partial_{x_1}^2+\partial_{x_2}^2,
$
we obtain
\begin{align}
\partial_{x_2}^2 w
=
w_{tt}
-
\partial_{x_1}^2 w
+
w.
\label{sn}
\end{align}
Applying
$
\partial_{x_1}^{\alpha}
$
to \eqref{sn} gives
\begin{align}
\partial_{x_1}^{\alpha}\partial_{x_2}^2 w
=
\partial_{x_1}^{\alpha} w_{tt}
-
\partial_{x_1}^{\alpha+2}w
+
\partial_{x_1}^{\alpha} w.
\label{sn-aL}
\end{align}
Taking the $L^2(\Omega_s)$-norm in \eqref{sn-aL}, we get
\begin{align}
\lVert
\partial_{x_1}^{\alpha}\partial_{x_2}^2 w(t)
\rVert_{L^2(\Omega_s)}
&\le
\lVert
\partial_{x_1}^{\alpha} w_{tt}(t)
\rVert_{L^2(\Omega_s)}
+
\lVert
\partial_{x_1}^{\alpha+2}w(t)
\rVert_{L^2(\Omega_s)}
\nonumber\\
&\qquad
+
\lVert
\partial_{x_1}^{\alpha} w(t)
\rVert_{L^2(\Omega_s)}.
\label{sneLO}
\end{align}
For
$\alpha+2\le 3,
$
we have
$\alpha=0,1.$
Hence \eqref{LemmaA}, \eqref{TLO}, and \eqref{sneLO} imply
\begin{align}
\sup_{t\in[0,T]}
\lVert
\partial_{x_1}^{\alpha}\partial_{x_2}^2 w(t)
\rVert_{L^2(\Omega_s)}
\le C,
\qquad
\alpha=0,1.
\label{snclLP}
\end{align}
It remains to control
$
\partial_{x_2}^3 w.
$
Differentiating \eqref{sn} with respect to $x_2$, we obtain
\begin{align}
\partial_{x_2}^3 w
=
\partial_{x_2} w_{tt}
-
\partial_{x_1}^2\partial_{x_2} w
+
\partial_{x_2} w.
\label{thnl}
\end{align}
Taking the $L^2(\Omega_s)$-norm in \eqref{thnl} gives
\begin{align}
\lVert
\partial_{x_2}^3 w(t)
\rVert_{L^2(\Omega_s)}
&\le
\lVert
\partial_{x_2} w_{tt}(t)
\rVert_{L^2(\Omega_s)}
+
\lVert
\partial_{x_1}^2\partial_{x_2} w(t)
\rVert_{L^2(\Omega_s)}
\nonumber\\
&\qquad
+
\lVert
\partial_{x_2} w(t)
\rVert_{L^2(\Omega_s)}.
\label{thnle}
\end{align}
By \eqref{E2.130s} with $\alpha=0$, we have
\begin{align}
\sup_{t\in[0,T]}
\lVert
w_{tt}(t)
\rVert_{H^1(\Omega_s)}
\le C,
\label{H1}
\end{align}
which implies that
\begin{align}
\sup_{t\in[0,T]}
\lVert
\partial_{x_2} w_{tt}(t)
\rVert_{L^2(\Omega_s)}
\le C.
\label{normal}
\end{align}
In addition, from \eqref{onPL}, with $\alpha=2$ and $\alpha=0$, we obtain
\begin{align}
\sup_{t\in[0,T]}
\lVert
\partial_{x_1}^2\partial_{x_2} w(t)
\rVert_{L^2(\Omega_s)}
&\le C,
\label{ASL}
\\
\sup_{t\in[0,T]}
\lVert
\partial_{x_2} w(t)
\rVert_{L^2(\Omega_s)}
&\le C.
\label{BSL}
\end{align}
Using \eqref{ASL}, \eqref{BSL}, \eqref{thnle}, and \eqref{normal}, we obtain
\begin{align}
\sup_{t\in[0,T]}
\lVert
\partial_{x_2}^3 w(t)
\rVert_{L^2(\Omega_s)}
\le C.
\label{tncOL}
\end{align}
Combining
\eqref{TLO},
\eqref{onPL},
\eqref{snclLP},
and
\eqref{tncOL},
we obtain all third-order derivatives of $w$. Indeed, with $\alpha=3$,
\eqref{TLO} yields
$$
\partial_{x_1}^3 w\in L^\infty(0,T;L^2(\Omega_s)).
$$
From \eqref{onPL} with $\alpha=2$, we obtain
$$
\partial_{x_1}^2\partial_{x_2} w
\in
L^\infty(0,T;L^2(\Omega_s)).
$$
From \eqref{snclLP} with $\alpha=1$, it follows that
$$
\partial_{x_1}\partial_{x_2}^2 w
\in
L^\infty(0,T;L^2(\Omega_s)).
$$
From \eqref{tncOL}, we get
$$
\partial_{x_2}^3 w
\in
L^\infty(0,T;L^2(\Omega_s)).
$$
In addition, the lower-order derivatives required by the $H^3(\Omega_s)$-norm are
already controlled. Indeed, \eqref{TLO} with
$\alpha=0,1,2$
yields
$$
w,
\quad
\partial_{x_1}w,
\quad
\partial_{x_1}^2w
\in
L^\infty(0,T;L^2(\Omega_s)).
$$
Moreover, \eqref{onPL} with
$\alpha=0,1$
gives
$$
\partial_{x_2}w,
\quad
\partial_{x_1}\partial_{x_2}w
\in
L^\infty(0,T;L^2(\Omega_s)).
$$
Finally, \eqref{snclLP} with $\alpha=0$ gives
$$
\partial_{x_2}^2w
\in
L^\infty(0,T;L^2(\Omega_s)).
$$
Since every derivative
$
\partial_{x_1}^{\alpha}\partial_{x_2}^b w$
with
$
\alpha+b \le3
$
belongs to
$
L^\infty(0,T;L^2(\Omega_s))$,
we conclude that
\begin{align}
w
\in
L^\infty(0,T;H^3(\Omega_s)).
\label{finalOP}
\end{align}
 we next prove
$
w_t\in L^\infty(0,T;H^2(\Omega_s)).
$
It is enough to control all derivatives
$
\partial_{x_1}^{\alpha}\partial_{x_2}^b w_t
$
with
$
\alpha+b \le2
$
in
$
L^\infty(0,T;L^2(\Omega_s))
$.  From \eqref{BHs}, we have
\begin{align}
\sup_{t\in[0,T]}
\lVert
\partial_{x_1}^{\alpha} w_t(t)
\rVert_{L^2(\Omega_s)}
\le C,
\qquad
0\le \alpha\le s.
\label{tL}
\end{align}
From \eqref{LemmaA}, we also have
\begin{align}
\sup_{t\in[0,T]}
\lVert
\partial_{x_1}^{\alpha} w_t(t)
\rVert_{H^1(\Omega_s)}
\le C,
\qquad
0\le \alpha\le s-1,
\label{H1}
\end{align}
and which therefore yields,
\begin{align}
\sup_{t\in[0,T]}
\lVert
\partial_{x_1}^{\alpha}\partial_{x_2} w_t(t)
\rVert_{L^2(\Omega_s)}
\le C,
\qquad
0\le \alpha\le s-1.
\label{on}
\end{align}
It remains to control the pure second normal derivative
$
\partial_{x_2}^2 w_t
$.  Differentiating \eqref{1.0b} once in time gives
\begin{align}
w_{ttt}
-
\Delta w_t
+
w_t
=
0.
\label{en}
\end{align}
Using
$
\Delta=\partial_{x_1}^2+\partial_{x_2}^2,
$
we obtain
\begin{align}
\partial_{x_2}^2 w_t
=
w_{ttt}
-
\partial_{x_1}^2 w_t
+
w_t.
\label{snP}
\end{align}
Taking the $L^2(\Omega_s)$-norm in \eqref{snP}, we get
\begin{align}
\lVert
\partial_{x_2}^2 w_t(t)
\rVert_{L^2(\Omega_s)}
&\le
\lVert
w_{ttt}(t)
\rVert_{L^2(\Omega_s)}
+
\lVert
\partial_{x_1}^2 w_t(t)
\rVert_{L^2(\Omega_s)}
\nonumber\\
&\qquad
+
\lVert
w_t(t)
\rVert_{L^2(\Omega_s)}.
\label{sneW}
\end{align}
By \eqref{E2.130s} with $\alpha=0$, we have
\begin{align}
\sup_{t\in[0,T]}
\lVert
w_{ttt}(t)
\rVert_{L^2(\Omega_s)}
\le C.
\label{HL}
\end{align}
Moreover, \eqref{tL} with $\alpha=2$ and $\alpha=0$ yields
\begin{align}
\sup_{t\in[0,T]}
\lVert
\partial_{x_1}^2 w_t(t)
\rVert_{L^2(\Omega_s)}
\le C,
\label{WTA2}
\end{align}
and
\begin{align}
\sup_{t\in[0,T]}
\lVert
w_t(t)
\rVert_{L^2(\Omega_s)}
\le C.
\label{WTA0}
\end{align}
Substituting \eqref{HL}, \eqref{WTA2}, and \eqref{WTA0}
into \eqref{sneW}, we obtain
\begin{align}
\sup_{t\in[0,T]}
\lVert
\partial_{x_2}^2 w_t(t)
\rVert_{L^2(\Omega_s)}
\le C.
\label{snclU}
\end{align}
Combining
\eqref{WTA2},
\eqref{on},
and
\eqref{snclU},
we have controlled all second-order derivatives of $w_t$. Indeed, \eqref{WTA2} yields
$$
\partial_{x_1}^2 w_t
\in
L^\infty(0,T;L^2(\Omega_s)).
$$
From \eqref{on} with $\alpha=1$, we obtain
$$
\partial_{x_1}\partial_{x_2} w_t
\in
L^\infty(0,T;L^2(\Omega_s)).
$$
From \eqref{snclU}, we obtain
$$
\partial_{x_2}^2 w_t
\in
L^\infty(0,T;L^2(\Omega_s)).
$$
In addition, the lower-order derivatives required by the $H^2(\Omega_s)$-regularity
of $w_t$ are also controlled. Indeed, \eqref{tL} with
$\alpha=0,1$
yields
$$
w_t,
\quad
\partial_{x_1}w_t
\in
L^\infty(0,T;L^2(\Omega_s)).
$$
Moreover, \eqref{on} with $\alpha=0$ gives
$$
\partial_{x_2}w_t
\in
L^\infty(0,T;L^2(\Omega_s)).
$$
Since every derivative
$
\partial_{x_1}^{\alpha}\partial_{x_2}^b w_t
$
with
$
\alpha+b \le2
$
belongs to
$
L^\infty(0,T;L^2(\Omega_s)),
$
we conclude that
\begin{align}
w_t
\in
L^\infty(0,T;H^2(\Omega_s)).
\label{final}
\end{align}
\noindent
Finally, we prove
$
u\in L^2(0,T;H^3(\Omega_f)).
$
It is enough to control all derivatives
$
\partial_{x_1}^{\alpha}\partial_{x_2}^b u
$
with
$
\alpha+b \le3
$
in
$L^2(0,T;L^2(\Omega_f)).$
\noindent From \eqref{BHs}, we have
\begin{align}
\sup_{t\in[0,T]}
\lVert
\partial_{x_1}^{\alpha} u(t)
\rVert_{L^2(\Omega_f)}
\le C,
\qquad
0\le \alpha \le s.
\label{2.184ETan}
\end{align}
Since $T<\infty$, and $L^\infty(0,T;L^2(\Omega_f)) \hookrightarrow L^2(0,T;L^2(\Omega_f))$, \eqref{2.184ETan} implies
\begin{align}
\int_0^T
\lVert
\partial_{x_1}^{\alpha} u(t)
\rVert_{L^2(\Omega_f)}^2\,dt
\le C,
\qquad
0\le \alpha \le s.
\label{2.185ETan}
\end{align}
In addition, from \eqref{BHs}, we have
\begin{align}
\int_0^T
\lVert
\nabla\partial_{x_1}^\alpha u(t)
\rVert_{L^2(\Omega_f)}^2\,dt
\le C,
\qquad
0\le\alpha\le s,
\end{align}
and  which yields,
\begin{align}
\int_0^T
\lVert
\partial_{x_1}^{\alpha+1}u(t)
\rVert_{L^2(\Omega_f)}^2\,dt
&\le C,
\qquad
0\le \alpha \le s,
\\
\int_0^T
\lVert
\partial_{x_1}^{\alpha}\partial_{x_2}u(t)
\rVert_{L^2(\Omega_f)}^2\,dt
&\le C,
\qquad
0\le \alpha \le s.
\label{2.188ENc}
\end{align}
Using \eqref{1.0a} and
$
\Delta=\partial_{x_1}^2+\partial_{x_2}^2,
$
we obtain
\begin{align}
\partial_{x_2}^2u
=
u_t-\partial_{x_1}^2u.
\label{NU2.189}
\end{align}
Applying
$
\partial_{x_1}^{\alpha}
$
to \eqref{NU2.189} gives
\begin{align}
\partial_{x_1}^{\alpha}\partial_{x_2}^2u
=
\partial_{x_1}^{\alpha} u_t
-
\partial_{x_1}^{\alpha+2}u.
\label{2.190Wf}
\end{align}
Taking the $L^2(\Omega_f)$-norm in \eqref{2.190Wf}, we get
\begin{align}
\lVert
\partial_{x_1}^{\alpha}\partial_{x_2}^2u(t)
\rVert_{L^2(\Omega_f)}
&\le
\lVert
\partial_{x_1}^{\alpha} u_t(t)
\rVert_{L^2(\Omega_f)}
+
\lVert
\partial_{x_1}^{\alpha+2}u(t)
\rVert_{L^2(\Omega_f)}.
\label{2.191ty}
\end{align}
For
$\alpha+2\le3$, we have
$\alpha=0,1.$
By \eqref{LemmaA}, we obtain
\begin{align}
\sup_{t\in[0,T]}
\lVert
\partial_{x_1}^{\alpha} u_t(t)
\rVert_{L^2(\Omega_f)}
\le C,
\qquad
\alpha=0,1.
\label{2.192TBs}
\end{align}
Since $T<\infty$, and $L^\infty(0,T;L^2(\Omega_f)) \hookrightarrow L^2(0,T;L^2(\Omega_f))$,  \eqref{2.192TBs} implies
\begin{align}
\int_0^T
\lVert
\partial_{x_1}^{\alpha} u_t(t)
\rVert_{L^2(\Omega_f)}^2\,dt
\le C,
\qquad
\alpha=0,1.
\label{2.194Qw}
\end{align}
Since
$
\alpha=0,1,
$
define
$
\beta=\alpha+2,
$
which implies
$
\beta=2,3.
$
Therefore, estimate \eqref{2.185ETan} implies that
$$
\int_0^T
\lVert 
\partial_{x_1}^{\beta}u(t)
\rVert_{L^2(\Omega_f)}^2\,dt
\le C,
\qquad
\beta=2,3.
$$
Because $s\ge3$, we have
$
2\le \beta \le s
$, it follows that
\begin{align}
\int_0^T
\lVert 
\partial_{x_1}^{\alpha+2}u(t)
\rVert_{L^2(\Omega_f)}^2\,dt 
\le C,
\qquad
\alpha=0,1.
\label{2.193Lz}  
\end{align}
Substituting \eqref{2.194Qw} and
\eqref{2.193Lz} into \eqref{2.191ty}, we obtain
\begin{align}
\int_0^T
\lVert
\partial_{x_1}^{\alpha}\partial_{x_2}^2u(t)
\rVert_{L^2(\Omega_f)}^2\,dt
\le C,
\qquad
\alpha=0,1.
\label{2.195bx}
\end{align}
It remains to control
$
\partial_{x_2}^3u.
$
Differentiating \eqref{NU2.189} with respect to $x_2$, we obtain
\begin{align}
\partial_{x_2}^3u
=
\partial_{x_2}u_t
-
\partial_{x_1}^2\partial_{x_2}u.
\label{Eq2.196}
\end{align}
Taking the $L^2(\Omega_f)$-norm in \eqref{Eq2.196}, we get
\begin{align}
\lVert
\partial_{x_2}^3u(t)
\rVert_{L^2(\Omega_f)}
\le
\lVert
\partial_{x_2}u_t(t)
\rVert_{L^2(\Omega_f)}
+
\lVert
\partial_{x_1}^2\partial_{x_2}u(t)
\rVert_{L^2(\Omega_f)}.
\label{Eq2.197w}
\end{align}
By \eqref{LemmaA} with $\alpha= 0$, we have
\begin{align}
\int_0^T
\lVert
\nabla u_t(t)
\rVert_{L^2(\Omega_f)}^2\,dt
\le C,
\end{align}
and therefore implies,
\begin{align}
\int_0^T
\lVert
\partial_{x_2}u_t(t)
\rVert_{L^2(\Omega_f)}^2\,dt
\le C.
\label{2.199DBe}
\end{align}
Moreover, from \eqref{2.188ENc} with $\alpha=2$, we obtain
\begin{align}
\int_0^T
\lVert
\partial_{x_1}^2\partial_{x_2}u(t)
\rVert_{L^2(\Omega_f)}^2\,dt
\le C.
\label{2.200Be}
\end{align}
Substituting \eqref{2.199DBe} and \eqref{2.200Be}
into \eqref{Eq2.197w}, we obtain
\begin{align}
\int_0^T
\lVert
\partial_{x_2}^3u(t)
\rVert_{L^2(\Omega_f)}^2\,dt
\le C.
\label{YU45qz}
\end{align}
Combining
\eqref{2.185ETan},
\eqref{2.188ENc},
\eqref{2.195bx},
and
\eqref{YU45qz},
we have controlled all third-order derivatives of $u$ in
$L^2(0,T;L^2(\Omega_f))$. Indeed, \eqref{2.185ETan} with $\alpha=3$ yields
$$
\partial_{x_1}^3u
\in
L^2(0,T;L^2(\Omega_f)).
$$
From \eqref{2.188ENc} with $\alpha=2$, we obtain
$$
\partial_{x_1}^2\partial_{x_2}u
\in
L^2(0,T;L^2(\Omega_f)).
$$
From \eqref{2.195bx} with $\alpha=1$, we obtain
$$
\partial_{x_1}\partial_{x_2}^2u
\in
L^2(0,T;L^2(\Omega_f)).
$$
From \eqref{YU45qz}, we get
$$
\partial_{x_2}^3u
\in
L^2(0,T;L^2(\Omega_f)).
$$
In addition, the lower-order derivatives required by the $H^3(\Omega_f)$-norm are
already controlled. Indeed, \eqref{2.185ETan} with
$\alpha=0,1,2$
yields
$$
u,
\quad
\partial_{x_1}u,
\quad
\partial_{x_1}^2u
\in
L^2(0,T;L^2(\Omega_f)).
$$
Moreover, \eqref{2.188ENc} with
$\alpha=0,1$
gives
$$
\partial_{x_2}u,
\quad
\partial_{x_1}\partial_{x_2}u
\in
L^2(0,T;L^2(\Omega_f)).
$$
Finally, \eqref{2.195bx} with $\alpha=0$ gives
$$
\partial_{x_2}^2u
\in
L^2(0,T;L^2(\Omega_f)).
$$
Since every derivative
$
\partial_{x_1}^{\alpha}\partial_{x_2}^b u
$
with
$\alpha+b \le3$
belongs to
$L^2(0,T;L^2(\Omega_f)),$
we conclude that
\begin{align}
u
\in
L^2(0,T;H^3(\Omega_f)).
\label{2.202}
\end{align}
The proof is complete.
\end{proof}
\noindent In addition to the above proof, we record the following observation concerning the regularity of the fluid component.
\subsection*{Further Regularity Observation for $
u\in L^\infty(0,T;H^2(\Omega_f)):
$}
\noindent The preceding argument proves
$
u\in L^2(0,T;H^3(\Omega_f)).$
However, it does not prove
$
u\in L^\infty(0,T;H^2(\Omega_f)).$
The reason is that the full
$
H^2(\Omega_f)$
norm contains the mixed derivative
$
\partial_{x_1}\partial_{x_2}u,$
and the available estimate controls this term only in
$
L^2(0,T;L^2(\Omega_f)),$
not in
$
L^\infty(0,T;L^2(\Omega_f)).$
This point is clarified in the following remark.

\begin{remark}
The estimates \eqref{BHs}, \eqref{LemmaA}, and \eqref{E2.130s} do not yield
$
u\in L^\infty(0,T;H^2(\Omega_f)).$
Indeed, \eqref{BHs} implies
$
\partial_{x_1}\partial_{x_2}u
\in
L^2(0,T;L^2(\Omega_f)),$
through the dissipation term
$$
\int_0^T
\lVert
\nabla D^\alpha u(t)
\rVert_{L^2(\Omega_f)}^2\,dt,
$$
but the available estimates do not provide
$
\partial_{x_1}\partial_{x_2}u
\in
L^\infty(0,T;L^2(\Omega_f)).$
Although \eqref{1.0a}:
$
u_t-\Delta u=0$
allows recovery of the pure normal derivative
$
\partial_{x_2}^2u$
through
$\partial_{x_2}^2u
=
u_t-\partial_{x_1}^2u,$
the mixed derivative
$
\partial_{x_1}\partial_{x_2}u$
cannot be recovered in
$L^\infty(0,T;L^2(\Omega_f)).$
Consequently, the preceding estimates close only the bound
$u\in L^2(0,T;H^3(\Omega_f)),$
while estimate
$u\in L^\infty(0,T;H^2(\Omega_f))$
cannot be deduced without additional regularity.
\end{remark}
\noindent \textbf{Declaration of competing interest}
\medskip

\noindent The author has no conflicts to disclose.

\medskip

\noindent \textbf{Author contributions}
\medskip

\noindent M. M. R. is the sole author of this manuscript and formulated the research idea, developed the mathematical theory, carried out the analysis, proved the main results, wrote the manuscript, and approved the final version.
\medskip

\noindent \textbf{Data availability}
\medskip

\noindent The data-sharing policy does not apply to this manuscript since we did not analyze any data in this study.
\bibliographystyle{plain}
\bibliography{reference}
\end{document}